\documentclass[11pt]{amsart}
\usepackage{float}
\usepackage[usenames]{color}
\usepackage{graphicx}
\usepackage{amssymb}
\usepackage{amscd}
\usepackage{makecell}
\usepackage{longtable}
\usepackage{graphicx}
\usepackage{epstopdf}
\usepackage{subfigure}
\usepackage{supertabular}

\theoremstyle{plain}
\newtheorem{thm}{Theorem}[section]
\newtheorem{lemma}[thm]{Lemma}

\newtheorem{prop}[thm]{Proposition}

\theoremstyle{definition}

\newtheorem{rmk}[thm]{Remark}
\newtheorem{example}[thm]{Example}

\def\dim{\mathop{\hbox {dim}}\nolimits}

\def\im{\mathop{\hbox {Im}}\nolimits}

\def\ker{\mathop{\hbox{Ker}}\nolimits}

\newcommand{\fra}{\mathfrak{a}}
\newcommand{\frb}{\mathfrak{b}}

\newcommand{\frg}{\mathfrak{g}}
\newcommand{\frh}{\mathfrak{h}}

\newcommand{\frk}{\mathfrak{k}}

\newcommand{\frp}{\mathfrak{p}}

\newcommand{\frt}{\mathfrak{t}}

\newcommand{\bbC}{\mathbb{C}}

\newcommand{\bbR}{\mathbb{R}}

\newcommand{\bbZ}{\mathbb{Z}}

\newcommand{\caC}{\mathcal{C}}

\begin{document}
	
	\title{Dirac series of $E_{7(7)}$}
	
	\author{Yi-Hao Ding}
	\address[Ding]{School of Mathematical Sciences, Soochow University, Suzhou 215006,
		P.~R.~China}
	\email{435025738@qq.com}

	\author{Chao-Ping Dong}
	\address[Dong]{School of Mathematical Sciences, Soochow University, Suzhou 215006,
		P.~R.~China}
	\email{chaopindong@163.com}

   \author{Lin Wei}
	\address[Wei]{School of Mathematical Sciences, Soochow University, Suzhou 215006,
		P.~R.~China}
	\email{1678728903@qq.com}

\abstract{This paper classifies all the Dirac series (that is, irreducible unitary representations having non-zero Dirac cohomology) of $E_{7(7)}$. Enhancing the Helgason-Johnson bound in 1969 for the group $E_{7(7)}$ is one key ingredient. Our calculation partially supports Vogan's fundamental parallelepiped (FPP) conjecture. As applications, when passing to Dirac index, we continue to find cancellation between the even part and the odd part of Dirac cohomology. Moreover, for the first time, we find Dirac series whose spin lowest $K$-types have multiplicities.}

\endabstract

	
	
	\maketitle

\section{Introduction}\label{sec-introduction}
This paper is a continuation of our study on the classification of \emph{Dirac series} (that is, irreducible unitary representations with non-zero Dirac cohomology) for  exceptional real reductive Lie groups \cite{DD,DDH,DDL,DDY}. The group that we shall focus on here is the linear split $E_7$.

Let us embark with necessary notations and background. Let $G$ be a connected simple Lie group with finite center. Let $\theta$ be the Cartan involution of $G$. Then $K:=G^{\theta}$ is a maximal compact subgroup of $G$, and
$$
\frg_0=\frk_0 \oplus \frp_0
$$
is the Cartan decomposition on the Lie algebra level. Let $T_f$ be a maximal torus of $K$. Let $\frg_0$, $\frk_0$, $\frt_{f, 0}$ be the Lie algebra of $G$, $K$, $T_f$, respectively. Let $\fra_{f, 0}=Z_{\frp_0}(\frt_{f, 0})$.
Then $\frh_{f, 0}=\frt_{f, 0}+\fra_{f, 0}$ is a maximally compact Cartan subalgebra of $\frg_0$. We shall drop the subscripts to stand for the corresponding complexified Lie algebras.

A fundamental problem in representation theory of Lie groups is classification of the \emph{unitary dual} $\widehat{G}$ of $G$. That is, the equivalence classes of all irreducible unitary representations. In the 1970s, Langlands classified the admissible dual $\widehat{G}_{\rm admi}$ of $G$.   Later, Knapp and Zuckerman classified the Hermitian dual $\widehat{G}_{\rm Herm}$ of $G$. It consists of members of  $\widehat{G}_{\rm admi}$ carrying an invariant Hermitian form. If  the Hermitian form is further required to be positive definite, then we come to the unitary dual $\widehat{G}$. Namely, we have
\begin{equation}\label{three-containment}
\widehat{G}\subset \widehat{G}_{\rm Herm} \subset \widehat{G}_{\rm admi}.
\end{equation}

Passing from the Hermitian dual to the unitary dual, one actually needs an effective way to rule out the great many non-unitary representations. Parthasarathy's Dirac operator inequality is a valuable tool on this aspect.
Fix a non-degenerate invariant bilinear form $B(\cdot,  \cdot)$ of $\frg$, which is positive definite on $\frp_0$ and negative definite on $\frk_0$. We also use the same symbol to denote its restrictions to other Lie subalgebras. Let $Z_1, \dots, Z_n$ be an orthonormal basis of $\frp_0$ with respect to $B(\cdot,  \cdot)$. Let $U(\frg)$ be the universal enveloping algebra of $\frg$. Let $C(\frp)$ be the Clifford algebra of $\frp$ with respect to $B(\cdot,  \cdot)$ .  Then the \emph{Dirac operator} introduced by Parthasarathy \cite{Pa1} is
$$
		D:=\sum_{i=1}^{n} Z_i\otimes Z_i\in U(\frg)\otimes C(\frp).
$$
This operator is well-defined: it does not depend on the choice of the orthogonal basis $\{Z_i\}_{i=1}^{n}$. Moreover, $D^2$ is a natural Laplacian,  and writing out the details carefully (which can be found for instance in [HP2]) will lead to Parthasarathy's Dirac operator inequality \cite{Pa2}:
Let $\pi$ be an irreducible unitary $(\frg, K)$ module with infinitesimal character $\Lambda$, then
\begin{equation}\label{Dirac-inequality-original}
	\|\gamma+\rho_c\|\geq \|\Lambda\|.
\end{equation}
Here $S_G$ is a spin module for the Clifford algebra $C(\frp)$, and $\gamma$ is the highest weight of any $\widetilde{K}$-type occurring in $\pi \otimes S_G$, and
    $$
    \widetilde{K}=\{(k,s) \in K \times \text{Spin}(\frp_0) : Ad(k)= p(s) \}
    $$
with $p: {\rm Spin}(\frp_0)\to {\rm SO}(\frp_0)$ being the universal double covering map.
In practice, whenever one can cleverly find a $\widetilde{K}$-type $\gamma$ such that \eqref{Dirac-inequality-original}
fails, one can immediately conclude that $\pi$ is \emph{not} unitary. This powerful inequality will be rephrased in Section \ref{sec-PRV-spin-DI} using the notion of spin norm, and when combined with Vogan pencil (see Section \ref{sec-pencil}), it will be more effective.

Let $\pi$ be a $(\frg, K)$-module, then $\pi \otimes S_G$ is a $(U(\frg) \otimes C(\frp), \widetilde{K})$-module, where $\widetilde{K}$ acts on $\pi$ through $K$ and acts on $S_G$ through Spin($\frp_0$). The Dirac operator $D\in U(\frg) \otimes C(\frp)$ acts on $\pi\otimes S_G$ in the obvious way. To further sharpen the Dirac inequality,  Vogan introduced Dirac cohomology of $\pi$ \cite{Vog97} as
\begin{equation}\label{Dirac-cohomology}
		H_D(\pi)=\ker D/\ker D \cap {\rm Im}\, D.
\end{equation}
Note that when $\pi$ is unitary, the case that we care the most, $\ker D$ and $\im D$ intersect trivially and the Dirac cohomology of $\pi$ becomes $\ker D=\ker D^2$.
	
	 Vogan conjectured that the Dirac cohomology, if nonzero, should refine the infinitesimal character.  The conjecture was proven by Huang and Pand\v zi\'c in 2002 \cite{HP}.
	
\begin{thm}{\rm (Theorem 2.3 of \cite{HP})}\label{thm-HP}
Let $\pi$ be an irreducible ($\frg$, $K$) module with infinitesimal character $\Lambda$.
Assume that $H_D(\pi)\neq 0$, and let $\gamma\in\frt_f^{*}\subset\frh_f^{*}$ be the highest weight of any $\widetilde{K}$-type in it. Then $\Lambda=w(\gamma+\rho_{c})$ for some element $w\in W(\frg,\frh_f)$.
\end{thm}
	
In the above theorem, we regard $\frt_f^*$ as a subspace of $\frh_f^*$ by extending the linear functionals on $\frt_f$ to be zero on $\fra_f$. Theorem \ref{thm-HP} says that Dirac cohomology, whenever non-zero, is a finer invariant of $\pi$ than infinitesimal character.  Let us denote by $\widehat{G}^d$ the Dirac series of $G$. Now conceptually, we have
\begin{equation}\label{four-containment}
\widehat{G}^d \subset\widehat{G}\subset \widehat{G}_{\rm Herm} \subset \widehat{G}_{\rm admi}.
\end{equation}
As explained above (see also Section \ref{sec-PRV-spin-DI} for details), Dirac series are members of $\widehat{G}$ such that Dirac inequality holds on certain $\widetilde{K}$-types $\gamma$ of $\pi\otimes S_G$.
Therefore, we may expect that $\widehat{G}^d$ cut out an interesting part of $\widehat{G}$. On the other hand, when $G$ is fixed, without knowing the entire unitary dual of $G$, Theorem A of \cite{D17} gives an algorithm to pin down $\widehat{G}^d$ via a finite calculation. Understanding $\widehat{G}^d$ will then help people to approximate $\widehat{G}$ from its interior.

The above are the motivations for us to classify $\widehat{G}^d$ for exceptional Lie groups.
One tool that we shall use is the software \texttt{atlas} \cite{At}, which detects unitarity based on the algorithm of Adams, van Leeuwen, Trapa and Vogan \cite{ALTV}. Firstly, let us briefly introduce it. Let $G(\bbC)$ be a complex connected  simple algebraic group with finite center. Let $\sigma$ be a \emph{real form} of $G(\bbC)$. That is $\sigma$ is an antiholomorphic Lie group automorphism of $G(\bbC)$ and $\sigma^2$=Id. Let $G=G(\bbC)^{\sigma}$. Let $\theta$ be the involutive algebraic automorphism of $G(\bbC)$ corresponding to $\sigma$ via the Cartan theorem (see for instance Theorem 3.2 of \cite{ALTV}). Let $K(\bbC):=G(\bbC)^{\theta}$. Let $H(\bbC)$ be a maximal torus of $G(\bbC)$. Its \emph{character lattice} is the group of algebraic homomorphisms
$$
X^*:=\text{Hom}_{alg}(H(\bbC),\bbC^{\times}).
$$
Choose a Borel subgroup $B(\bbC) \supset H(\bbC)$. Let $\frb$ (resp., $\frh$) be the Lie algebra of $B(\bbC)$ (resp., $H(\bbC)$). Let $l=\dim_{\bbC} \frh$ and let $0, 1, \dots, l-1$ be a labelling of the simple roots for $\Delta(\frb, \frh)$.

In \texttt{atlas}, an irreducible $(\frg, K)$ module $\pi$ is parameterized by a \emph{final} parameter $p=(x,\lambda,\nu)$ via the Langlands classification, where $x$ is a $K(\bbC)$-orbit of the Borel variety $G(\bbC)/B(\bbC)$, $\lambda \in X^*+\rho$ and $\nu \in (X^*)^{-\theta} \otimes_{\bbZ} \bbC$. The support of $x$ is a subset of $[0, 1, \dots, l-1]$ and it is given by the \texttt{atlas} command \texttt{support(x)}.  The Cartan involution $\theta$ now becomes $\theta_x$, and the latter is given by the command \texttt{involution(x)}. We say the representation $\pi$ is \emph{fully supported} if \texttt{support(x)} contains all the simple roots.
In this setting, the infinitesimal character of $\pi$ is
\begin{equation}\label{inf-char}
	\frac{1}{2}(1+\theta)\lambda +\nu \in\frh^*
\end{equation}
All the finitely many irreducible $(\frg, K)$ modules with infinitesimal character \texttt{Lambda} can be constructed via the command
\begin{verbatim}
	set all=all_parameters_gamma(G, Lambda)
\end{verbatim}
\texttt{atlas} lists these representations according to cardinality of their supports. We also adopt this smart way to organize the Dirac series of $G$. In particular, we divide $\widehat{G}^d$ into two parts: those which are fully supported, and those which are not. The first part is finite and we call them \emph{FS-scattered} representations. The second part is infinite, but can be arranged into finitely many \emph{strings}, with the representations in each string having the same KGB element $x$ and the same parameter $\nu$. We will recall the relevant aspects of \texttt{atlas} whenever needed.

Our main result is stated as follows.

\begin{thm}\label{thm-main}
The set $\widehat{E_{7(7)}}^d$ consists of $125$ FS-scattered representations whose spin lowest $K$-types are u-small, and $2057$ strings of representations.
\end{thm}

The notion of u-small $K$-type, due to Salamanca-Riba and Vogan \cite{SV}, will be recalled in Section \ref{sec-u-small}.
One key ingredient leading to Theorem \ref{thm-main} is a further improvement of the Helgason-Johnson bound \cite{HJ} for $E_{7(7)}$. See Proposition \ref{prop-HJ}. It is worth mentioning that, unlike all the previously known cases, on the group $E_{7(7)}$, we find Dirac series whose spin lowest $K$-types have multiplicities. This should be the first example in the literature. See Remark \ref{rmk-mult} for more.

The paper is organized as follows: Section \ref{sec-EV-structure} recalls the basic structure of $E_{7(7)}$ and necessary preliminaries. Section \ref{sec-HJ} further sharpens the Helgason-Johnson bound for $E_{7(7)}$.  We classify the Dirac series of $E_{7(7)}$ in Section \ref{sec-EV-ds}, and study the Dirac index of some FS-scattered members in Section \ref{sec-EV-DI}. Section \ref{sec-two-ds} carefully examines the multiplicities of the spin LKTs of two FS-scattered representations. Section \ref{sec-appendix} is an appendix presenting all the FS-scattered representations according to their infinitesimal characters.

\section{Basic structure of $E_{7(7)}$}\label{sec-EV-structure}
We continue with the notation of the previous section. From now on, we fix $G$ as the connected simple real exceptional linear Lie group \texttt{E7\_s} in \texttt{atlas}. It has center $\bbZ/2\bbZ$. The group $G$ is not simply connected. This can be explained as follows:
\begin{verbatim}
G:E7_s
set K=K_0(G)
K
Value: compact connected real group with Lie algebra 'su(8)'
print_Z(K)
Group is semisimple
center=Z/4Z
\end{verbatim}
The output says that $K$ has center $\bbZ/4\bbZ$. Therefore, its universal covering group $\widetilde{K}\cong SU(8)$, which has center $\bbZ/8\bbZ$, is a double cover of $K$. Using the Cartan decomposition $G=K \exp(\frp_0)$, we know that the universal covering group $\widetilde{G}$ is actually a double cover of $G$. Classifying the Dirac series of $\widetilde{G}$ is a valuable project, since it will offer us some genuine representations of $\widetilde{G}$. However, since \texttt{atlas} can not handle non-linear groups at the current stage, let us leave it to the future.

Note that $G$ is equal rank. That is, $\frh_f=\frt_f$. The Lie algebra $\frg_0$ of $G$ is denoted as \texttt{EV} in \cite[Appendix C]{Kn}.
	Note that
	$$
	-\dim \frk +\dim \frp=-63+70=7.
	$$
	Therefore, the group $G$ is also called $E_{7(7)}$ in the literature.
	
Let $\Delta(\frg,\frt_f)$, $\Delta(\frk,\frt_f)$ be the root systems. As usual, let $W(\frg,\frt_f)$, $W(\frk,\frt_f)$ be the corresponding Weyl groups.
	
We fix a Vogan diagram for $\frg_0$ in Figure \ref{V diagram}. By doing this, we have actually fixed a positive root system $\Delta^+(\frg,\frh_f)=(\Delta^+)^{(0)}(\frg,\frh_f)$ for $\Delta(\frg,\frh_f)$, with the corresponding simple roots $\alpha_1=\frac{1}{2}(1, -1,-1,-1,-1,-1,-1,1)$, $\alpha_2=e_1+e_2$ and $\alpha_i=e_{i-1}-e_{i-2}$ for $3\leq i\leq 7$. Let $\zeta_1, \dots, \zeta_7$ be the corresponding fundamental weights. We will use them as a basis to express the \texttt{atlas} parameters $\lambda$, $\nu$ and the infinitesimal character. More precisely, in such cases, $[a, b, c, d, e, f, g]$ stands for the vector $a\zeta_1+\cdots+ g \zeta_7$.

	\begin{figure}[h]
		\begin{center}
			\includegraphics[scale=0.6]{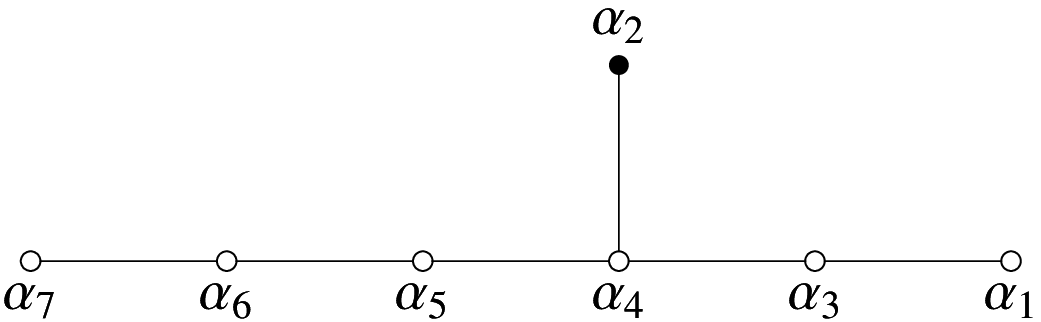}
		\end{center}
		\caption{The Vogan diagram for EV}
		\label{V diagram}
	\end{figure}

We fix a Dynkin diagram for $\Delta^+(\frk, \frt_f)$ in Figure \ref{D diagram}, with the simple roots being $\gamma_1=\alpha_1$, $\gamma_i=\alpha_{i+1}$ for $2 \leq i \leq 6$ and $\gamma_7=\alpha_1+2\alpha_2+2\alpha_3+3\alpha_4+2\alpha_5+\alpha_6$. Let $\varpi_1, \dots, \varpi_7\in \frt_f^*$ be the corresponding fundamental weights.

	\begin{figure}[h]
		\begin{center}
			\includegraphics[scale=0.6]{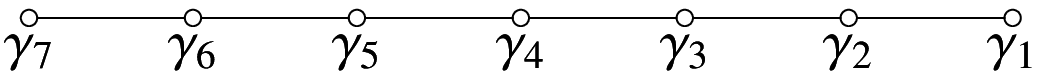}
		\end{center}
		\caption{The Dynkin diagram for $\Delta^+(\frk,\frt_f)$}
		\label{D diagram}
	\end{figure}
	
\subsection{Positive root systems and cones}
Including $(\Delta^{+})^{(0)}(\frg,\frt_f)$, there are $72$ choices of the positive root system of $\Delta(\frg,\frt_f)$ containing the fixed $\Delta^{+}(\frk,\frt_f)$. We enumerate them as
$$
(\Delta^{+})^{(j)}(\frg,\frt_f)=\Delta^{+}(\frk,\frt_f) \cup (\Delta^{+})^{(j)}(\frp, \frt_f),
\quad 0 \leq j \leq 71.
$$
Let us denote by $\rho^{(j)}$ (resp., $\rho_n^{(j)}$, $\rho_c$) the half sum of roots in
	 $(\Delta^{+})^{(j)}(\frg, \frt_f)$ (resp., $(\Delta^+)^{(j)}(\frp, \frt_f)$, $\Delta^+(\frk, \frt_f)$). Then we have
	 $$
	 \rho^{(j)}=\rho_n^{(j)}+\rho_c.
	 $$
Let $w^{(j)}$ be the unique element in $W(\frg,\frt_f)$ such that $w^{(j)}\rho^{(0)}=\rho^{(j)}$. Then, for  $0\leq j\leq 71$, $w^{(j)}\alpha_1, \dots, w^{(j)}\alpha_7$ are the simple roots of $(\Delta^+)^{(j)}(\frg, \frt_f)$, with $w^{(j)}\zeta_1, \dots, w^{(j)}\zeta_7$ being the corresponding fundamental weights. Denote by
	 $$
     W(\frg, \frt_f)^1=\{w^{(j)}\mid 0\leq j\leq 71\}.
     $$
	
Let $\mathcal{C}_{\frg}^{(j)}$ (resp., $\mathcal{C}_{\frk}$) be the dominant Weyl chamber corresponding to $(\Delta^{+})^{(j)}(\frg, \frt_f)$ (resp., $\Delta^{+}(\frk, \frt_f)$). Then $w^{(j)}\mathcal{C}_{\frg}^{(0)}=\mathcal{C}_{\frg}^{(j)}$, $\mathcal{C}_{\frk}=\bigcup_{j=0}^{71}\mathcal{C}_{\frg}^{(j)}$, and
$$
W(\frg, \frt_f)^1=\{ w \in W(\frg,\frt_f)| w(\mathcal{C}_{\frg}) \subseteq \mathcal{C}_{\frk} \}.
$$
By Kostant \cite{Ko}, the multiplication map gives a bijection from $W(\frg, \frt_f)^1\times W(\frk, \frt_f)$ onto $W(\frg, \frt_f)$. In particular,
	$$
	|W(\frg, \frt_f)^1|=\frac{|W(\frg, \frt_f)|}{|W(\frk, \frt_f)|}=\frac{2903040}{40320}=72.
	$$

\subsection{$\frk$-types and $K$-types}\label{sec-k-type-K-type}
By the highest weight theorem, any irreducible $\frk$-module ($\frk$-\emph{type} for short henceforth) can be parameterized by its highest weight $\mu=a \varpi_1+b \varpi_2 + c \varpi_3+d \varpi_4 +e \varpi_5+ f\varpi_6+ g\varpi_7$, where the coefficients $a, \dots, g$ run over non-negative integers. For simplicity, we shall denote this $\frk$-type by $E_{\mu}$, and denote $\mu$ as $[a, b, c, d, e, f, g]$. Let \begin{equation}\label{beta}
\beta=\alpha_1+\alpha_2+2\alpha_3+3\alpha_4+3\alpha_5+2\alpha_6+\alpha_7=[0,0,0,1,0,0,0].
\end{equation}
Then $\frp\cong E_{\beta}$ as $\frk$-modules. The $\frk$-types $E_{[a,b,c,d,e,f,g]}$ and $E_{[g,f,e,d,c,b,a]}$ are contragredient to each other.

Note that a $\frk$-type $E_{[a, b, c, d, e, f, g]}$ is  a $K$-type if and only if
$a+c+e+g$ is even. For instance,
\begin{equation}\label{rhoc}
\rho_c=(-\frac{5}{2},-\frac{3}{2},-\frac{1}{2},\frac{1}{2},\frac{3}{2},\frac{5}{2},-\frac{7}{2},\frac{7}{2})=[1, 1, 1, 1, 1, 1, 1]
\end{equation}
is the highest weight of a $K$-type. Note that
\begin{equation}\label{rhoc-g}
\rho_c=\zeta_1-4\zeta_2+ \zeta_3+ \zeta_4+  \zeta_5+ \zeta_6 +\zeta_7.
\end{equation}
	
Later, we shall often need to shift the \texttt{atlas} coordinates of a $K$-type to the coordinates using the fundamental weights $\varpi_1, \dots, \varpi_7$ as a basis. As learned from Vogan, there is a clever choice of the KGB element which makes the job much easier. Let us illustrate this for \texttt{E7\_s} (certain outputs are omitted).
	\begin{verbatim}
	G:E7_s
	void: for x in distinguished_fiber(G) do prints(x," ",rho_c(x)) od
	KGB element #0 [1, 0, 0, 1, 0, 1, 0]/1
	KGB element #1 [-1, 0, 1, 1, 0, 1, 0]/1
	......
	KGB element #70 [1, 4, 1, -3, 1, 1, 1]/1
	KGB element #71 [1, -4, 1, 1, 1, 1, 1]/1
\end{verbatim}
Therefore, as suggested by \eqref{rhoc-g}, we will choose \texttt{KGB(G,71)} to print $K$-types. Then, if a $K$-type has \texttt{atlas} coordinates $(y_1, y_2, y_3, y_4, y_5, y_6, y_7)$, its coordinates in terms of $\varpi_1, \dots, \varpi_7$ are
\begin{equation}\label{K-coords-shift}
[y_1, y_3, y_4, y_5, y_6, y_7, y_1 + 2 y_2 + 2 y_3 + 3 y_4 + 2 y_5 + y_6].
\end{equation}

Let us see an example.
\begin{verbatim}
set p=parameter(KGB(G,20925),[1,1,1,1,1,1,1]/1,[1,1,1,0,1,1,1]/1)
height(p)
Value: 0
print_branch_irr(p,KGB(G,71),100)
(1+0s)*(KGB element #71,[0, 0, 0, 0, 0, 0, 0])
(1+0s)*(KGB element #71,[0, -1, 0, 0, 1, 0, 0])
(1+0s)*(KGB element #71,[0, -2, 0, 0, 2, 0, 0])
(1+0s)*(KGB element #71,[0, -3, 0, 0, 3, 0, 0])
\end{verbatim}
By \eqref{K-coords-shift}, the above four  $K$-types are  $[0,0,0,n,0,0,0]=n\beta$, where $0\leq n\leq 3$.

\subsection{Lambda norm, lowest $K$-types and infinitesimal character}\label{sec-LKT}
Fix a $K$-type $\mu$. Choose an index $0\leq j\leq 71$ such that $\mu+2\rho_c\in \caC_{\frg}^{(j)}$.  Define
    \begin{equation}\label{lambda-a-mu}
	\lambda_a(\mu):=P(\mu+2 \rho_c -\rho^{(j)}),
    \end{equation}
where $P(\cdot)$ is the projection of $\mu+2 \rho_c -\rho^{(j)}$ onto the cone $\caC_{\frg}^{(j)}$.
    It turns out $\lambda_a(\mu)$ is well-defined. That is, it is independent of the choice of an allowable  $j$. The \emph{lambda norm} of $\mu$ is defined \cite{Vog81, Ca} as
    \begin{equation}\label{lambda-norm}
    	\|\mu\|_{\rm lambda}:=\|\lambda_a(\mu)\|.
    \end{equation}
    Now let $\pi$ be an irreducible $(\frg, K)$ module. A $K$-type $\mu$ is called a \emph{lowest $K$-type} (LKT for short) of $\pi$ if $\mu$ occurs in $\pi$ and $\|\mu\|_{\rm lambda}$ attains the minimum among all the $K$-types of $\pi$. It is easy to see that $\pi$ has finitely many LKTs.

    The lowest $K$-types and the infinitesimal character are both important invariants of $\pi$. Moreover, they are linked in the following way: Let $\mu$ be one of the LKTs of $\pi$, then a representative of the infinitesimal character of $\pi$ can be chosen as
    \begin{equation}\label{inf-char}
    	\Lambda=(\lambda_a(\mu), \nu)\in \frh^*=\frt^*+\fra^*.
    \end{equation}
    Note that $\nu$ in \eqref{inf-char} has the same norm as the $\nu$-part in the Langlands parameter of $\pi$. Abusing the notation a bit, we  will not distinguish them.

\subsection{PRV component, spin norm and Dirac inequality}\label{sec-PRV-spin-DI}
Let $E_\mu$ and $E_\nu$ be two $\frk$-types. By  \cite{PRV}, the $\frk$-type $E_{\{\mu+w_0^K\nu\}}$ occurs exactly once in $E_\mu \otimes E_\nu$, where $w_0^K$ is the longest element in $W(\frk, \frt_f)$ and $\{\mu+w_0^K\nu\}$ is the unique element in $\mathcal{C}_{\frk}$ to which $\mu+w_0^K\nu$ is conjugate under the action of $W(\frk,\frt_f)$.
	
For the group $E_{7(7)}$, we have that
\begin{equation}\label{spin-module}
S_G=\bigoplus_{j=0}^{71} E_{\rho_n^{(j)}}
\end{equation}
as $\frk$-modules.     So the PRV components for $E_\mu \otimes S_G$ are $\{\mu+w_0^K\rho_n^{(j)}\mid 0 \leq j \leq 71\}$. Although $w_0^K\neq -1$, it turns out that these components coincide with $\{\mu-\rho_n^{(j)}\mid 0 \leq j \leq 71\}$ since $S_G$ is self-dual as a $\frk$ module.

The \emph{spin norm} of a $K$-type $E_{\mu}$ \cite{D13} is
    \begin{equation}\label{spin-norm}
    	\|\mu\|_{\rm spin} :=\min_{0\leq j \leq 71}\|\{\mu-\rho_n^{(j)}\}+\rho_c\|.
    \end{equation}
The \emph{spin norm} of an irreducible $(\frg, K)$ module $\pi$ is defined as
    $$
    \|\pi\|_{\rm spin}=\min \|\delta\|_{\rm spin},
    $$
    where $\delta$ runs over all the $K$-types of $\pi$. We call the $K$-type $\delta$ a \emph{spin lowest $K$-type} (spin LKT for short) of $\pi$ if it occurs in $\pi$ and that $\|\delta\|_{\rm spin}=\|\pi\|_{\rm spin}$.

    When $\pi$ is unitary and has infinitesimal character $\Lambda$, the original Dirac operator inequality \eqref{Dirac-inequality-original} can be rephrased as
    \begin{equation}\label{Dirac-inequality}
    	\|\pi\|_{\rm spin}\geq \|\Lambda\|.
    \end{equation}
    By Theorem 3.5.12 of \cite{HP2}, equality happens in \eqref{Dirac-inequality} if and only if $\pi$ has non-zero Dirac cohomology. Moreover, in such a case, it is exactly the spin LKTs of $\pi$ that contribute to $H_D(\pi)$.

\subsection{The u-small convex hull}\label{sec-u-small}
    Salamanca-Riba and Vogan introduced unitarily-small $K$-types to describe their unified conjecture about the unitary dual of real reductive Lie groups in \cite{SV}. Let
    $$
    R(\Delta(\frp,\frt_f))=\{\sum_{\alpha \in \Delta(\frp,\frt_f)}b_{\alpha}\alpha|0 \leq b_{\alpha} \leq 1\},
    $$
which is invariant under $W(\frk,\frt_f)$. For the group $E_{7(7)}$, this is the convex hull formed by the $W(\frk, \frt_f)$ orbits of the weights $2\rho_n^{(j)}$, $0\leq j\leq 71$. As in \cite{SV}, we  call $R(\Delta(\frp,\frt_f))$ the \emph{unitarily small} (\emph{u-small} for short) \emph{convex hull}. A $K$-type (or $\frk$-type) is called \emph{u-small} if its highest weight lies in the u-small convex hull. Otherwise, we will say it is \emph{u-large}.

\subsection{Vogan pencil}\label{sec-pencil}
The group $E_{7(7)}$ is not Hermitian symmetric.  Therefore, by Lemma 3.4 of \cite{Vog80}, the $K$-types of any infinite-dimensional $(\frg, K)$ module $\pi$ must be the union of certain \emph{pencils} $P(\mu)$, where
$$
P(\mu):=\{\mu+n\beta|n \in \mathbb{Z}_{\geq 0}\}.
$$

Our experience is that if the $K$-type $\mu$ is u-small, the spin norm should decrease along the Vogan pencil $P(\mu)$ firstly, and then strictly increases. This has been proven in some cases including $E_{7(7)}$
\cite{D17}. Thus using Dirac inequality along $P(\mu)$ will be more powerful in testing non-unitarity than only using it for the single $K$-type $\mu$.

\section{The Helgason-Johnson bound for $E_{7(7)}$}\label{sec-HJ}

\begin{prop}\label{prop-previous-HJ}
Let $G$ be $E_{7(7)}$. Let $\pi$ be an irreducible unitary $(\frg, K)$ module  whose infinitesimal character $\Lambda$ is given by \eqref{inf-char}. Then
\begin{itemize}
\item [(a)] $\|\nu\|\leq \sqrt{\frac{399}{2}}=\|\rho(G)\|$;
\item [(b)] $\|\nu\|\leq \sqrt{\frac{231}{2}}$ if $\pi$ is infinite-dimensional.
\end{itemize} 	
\end{prop}

Item (a) of the above proposition traces back to Helgason and Johnson \cite{HJ} in 1969. Item (b) is obtained recently in \cite{D20}. Improvement of this bound will save us a lot of time in doing non-unitarity test. This section aims to enhance it further by adopting suitable assumptions.

\begin{example}\label{exam-HJ}
One can check directly that the following two irreducible representations are unitary:
\begin{verbatim}
final parameter(x=20925,lambda=[1,1,1,1,1,1,1]/1,nu=[1,1,1,0,1,0,1]/1)
final parameter(x=20925,lambda=[1,1,1,1,1,1,1]/1,nu=[1,1,1,0,1,1,1]/1)
\end{verbatim}
Both of them have non-zero Dirac cohomology (see the first entries of Tables \ref{table-EV-1110101} and \ref{table-EV-1110111}, respectively). The statistic $\|\nu\|^2$ for them are $\frac{159}{2}$ and $\frac{231}{2}$, respectively. Moreover, they have GK dimension $26$ and $17$, respectively. Indeed, the latter one is the unique minimal representation of $E_{7(7)}$.
\end{example}

\begin{prop}\label{prop-HJ}
Let $G$ be $E_{7(7)}$. Except for the two representations described in Example \ref{exam-HJ}, there is no irreducible unitary $(\frg, K)$ module $\pi$ whose infinitesimal character $\Lambda$, as given by \eqref{inf-char}, is a non-negative integer combination of $\zeta_1, \dots, \zeta_7$ such that
\begin{equation}\label{further-HJ-bound}
\frac{157}{2}	\leq \|\nu\|^2 \leq \frac{231}{2}.
\end{equation}		
\end{prop}
	
	\begin{proof}
		Take a LKT $\mu$ of $\pi$. Then its infinitesimal character has the form $(\lambda_a(\mu), \nu)$ as in \eqref{inf-char}. By the Dirac inequality \eqref{Dirac-inequality}, one has that
		$$
		\|\Lambda\|^2=\|\lambda_a(\mu)\|^2 +\|\nu\|^2\leq \|\mu\|^2_{\rm spin}.
		$$
		Therefore,
		\begin{equation}\label{nu-bound}
			\|\nu\|^2\leq \|\mu\|^2_{\rm spin}- \|\mu\|^2_{\rm lambda}.
		\end{equation}
		As computed in Section 5.7 of \cite{D20}, we have that $\max\,\{A_j\mid 0\leq j\leq 71\}=78$. We refer the reader to Section 3 of \cite{D20} for the precise meaning of these $A_j$. It follows that for any u-large $K$-type $\mu$, one has that
		$$
		\|\mu\|_{\rm spin}^2 - \|\mu\|_{\rm lambda}^2\leq 78.
		$$
		Since $\|\nu\|^2\geq 78.5$ by our assumption \eqref{further-HJ-bound}, we conclude from \eqref{nu-bound} that $\mu$ can \emph{not} be u-large.
		
		There are $97752$ u-small $K$-types in total. Among them, only $61$ have the property that
		$$
		78.5\leq \|\mu\|_{\rm spin}^2 - \|\mu\|_{\rm lambda}^2.
		$$
		Let us collect these $61$ u-small $K$-types as \texttt{Certs}.

		Moreover, we compute that
		$0\leq \|\lambda_a(\mu)\|^2\leq 24.5$ for any $\mu\in \texttt{Certs}$. Therefore,
		\begin{equation}\label{can-Lambda}
			0+ 78.5 \leq \|\Lambda\|^2=\|\lambda_a(\mu)\|^2+\|\nu\|^2\leq 24.5 +115.5.
		\end{equation}		
		The right hand side above uses item (b) of Proposition \ref{prop-previous-HJ}. There are $1172$ integral $\Lambda$s meeting the requirement \eqref{can-Lambda}. We collect them as $\Omega$.

		Now a direct search using \texttt{atlas} says that there are $33594$ irreducible  representations $\pi$ such that $\Lambda\in \Omega$ and that $\pi$ has a LKT which is a member of \texttt{Certs}. It turns out that only two of them are unitary. They are described in Example \ref{exam-HJ}. This finishes the proof.	
	\end{proof}

\section{Dirac series of $E_{7(7)}$}\label{sec-EV-ds}

This section reports the Dirac series of $E_{7(7)}$. As learned from \texttt{atlas}, we will divide them into two parts according to whether the KGB element is fully-supported or not. The two parts are called FS-scattered representations and string representations, respectively.
Theorem A of \cite{D17} guarantees that there are only \emph{finitely many} FS-scattered representations which should be viewed as the core of Dirac series. Moreover, \cite{D21} gives a method counting the number of strings. Here the infinitely many Dirac series sharing the same KGB element \texttt{x}, which is not fully supported, and the same parameter $\nu$ is viewed as a \emph{string}.
	
\subsection{FS-scattered representations of $E_{7(7)}$}\label{sec-FS-EV}
	This subsection aims to sieve out the FS-scattered Dirac series representations for $E_{7(7)}$. Except for the trivial representation, and the two representations in Example \ref{exam-HJ}, Propositions \ref{prop-previous-HJ} and \ref{prop-HJ} tell us that we can find the remaining FS-scattered representations in the following way: enumerate the infinitesimal characters $\Lambda=[a, b, c, d, e, f, g]$ such that
	\begin{itemize}
		\item[$\bullet$] $a$, $b$, $c$, $d$, $e$, $f$, $g$ are non-negative integers;
		\item[$\bullet$] $a+c>0$, $b+d>0$, $c+d>0$, $d+e>0$, $e+f>0$, $f+g>0$;
		\item[$\bullet$] $\min\{a, b, c, d, e, f, g\}=0$;
		\item[$\bullet$] there exists a fully supported KGB element $x$ such that $\|\frac{\Lambda-\theta_x\Lambda}{2}\|< \sqrt{\frac{157}{2}}$.
	\end{itemize}
 The third item above uses the main result of Salamanca-Riba \cite{Sa}. For the fourth item, let $p$ be any irreducible representation  of $E_{7(7)}$ which has \texttt{atlas} parameter $(x,\lambda,\nu)$. Then we have
 $$\nu=\frac{\Lambda-\theta_x(\Lambda)}{2},$$
where $\theta_x$ is \texttt{involution(x)} in \texttt{atlas}. Therefore, the fourth item actually means
 \begin{equation}\label{new-HJ-bound}
 \|\nu\|<\sqrt{\frac{157}{2}}.
 \end{equation}
The second item is explained below.

	\begin{lemma}\label{lemma-EV-HP}
	Let $\Lambda=a\zeta_1+b\zeta_2+c\zeta_3+d\zeta_4+e\zeta_5+f \zeta_6+g \zeta_7$ be the infinitesimal character of any Dirac series representation $\pi$ of $E_{7(7)}$ which is dominant with respect to $\Delta^+(\frg, \frt_f)$. Then $a$, $b$, $c$, $d$, $e$, $f$, $g$ must be non-negative integers such that $a+c>0$, $b+d>0$, $c+d>0$, $d+e>0$, $e+f>0$, $f+g>0$.
     \end{lemma}
	
	\begin{proof}
		Since the group \texttt{E7\_s} is linear, it follows from Remark 4.1 of \cite{D21} that $a$, $b$, $c$, $d$, $e$, $f$, $g$ must be non-negative integers.
		
		Now if $a+c=0$, i.e., $a=c=0$,  a direct check says that for any $w\in W(\frg, \frt_f)^1$, at least one coordinate of $w\Lambda$ in terms of the basis $\{\varpi_1, \dots, \varpi_7\}$ vanishes. Therefore,
		$$
		\{\mu-\rho_n^{(j)}\} + \rho_c =w \Lambda
		$$
		could not hold for any $K$-type $\mu$. This proves that $a+c>0$. Other inequalities can be similarly deduced.
	\end{proof}

Let us collect all the infinitesimal characters satisfying the above four conditions as $\Phi$. We divide $\Phi$ into $\Phi_1,\dots,\Phi_{12}$, where $\Phi_i$ consists of the members of $\Phi$ whose largest coordinate equals  $i$. Note that $|\Phi|=271379$. The cardinalities of those $\Phi_i$s are presented as follows:
	\begin{center}
		\begin{tabular}{c|c|c|c|c|c}
			$\#\Phi_1$ & $\#\Phi_2$ & $\#\Phi_3$ & $\#\Phi_4$ & $\#\Phi_5$ & $\#\Phi_6$   \\
			\hline
			$35$ & $1085$  & $8518$ & $30459$ & $55014$ & $62169$ \\
			\hline
		$\#\Phi_7$ &	$\#\Phi_8$  & $\#\Phi_9$ & $\#\Phi_{10}$& $\#\Phi_{11}$ & $\#\Phi_{12}$  \\
			\hline
			 $51970$ &$34289$ & $18146$ & $7486$ & $2027$ & $181$
		\end{tabular}
	\end{center}

Although it only occupies $0.0128\%$ of $\Phi$, it turns out that the interesting part of the story happens within $\Phi_1$, whose members  are listed below:
	\begin{align*}
		&[0, 0, 1, 1, 0, 1, 0], [0, 0, 1, 1, 0, 1, 1], [0, 0, 1, 1, 1, 0, 1], [0, 0, 1, 1, 1, 1, 0], [0, 0, 1, 1, 1, 1, 1],\\
		&[0, 1, 1, 0, 1, 0, 1], [0, 1, 1, 0, 1, 1, 0], [0, 1, 1, 0, 1, 1, 1], [0, 1, 1, 1, 0, 1, 0], [0, 1, 1, 1, 0, 1, 1],\\
		& [0, 1, 1, 1, 1, 0, 1], [0, 1, 1, 1, 1, 1, 0], [0, 1, 1, 1, 1, 1, 1], [1, 0, 0, 1, 0, 1, 0], [1, 0, 0, 1, 0, 1, 1],\\
		& [1, 0, 0, 1, 1, 0, 1], [1, 0, 0, 1, 1, 1, 0], [1, 0, 0, 1, 1, 1, 1], [1, 0, 1, 1, 0, 1, 0], [1, 0, 1, 1, 0, 1, 1], \\
		&[1, 0, 1, 1, 1, 0, 1], [1, 0, 1, 1, 1, 1, 0], [1, 0, 1, 1, 1, 1, 1], [1, 1, 0, 1, 0, 1, 0], [1, 1, 0, 1, 0, 1, 1], \\
		& [1, 1, 0, 1, 1, 0, 1], [1, 1, 0, 1, 1, 1, 0], [1, 1, 0, 1, 1, 1, 1], [1, 1, 1, 0, 1, 0, 1], [1, 1, 1, 0, 1, 1, 0],	\\
& [1, 1, 1, 0, 1, 1, 1], [1, 1, 1, 1, 0, 1, 0],	[1, 1, 1, 1, 0, 1, 1], [1, 1, 1, 1, 1, 0, 1], [1, 1, 1, 1, 1, 1, 0].
	\end{align*}
	
	Let $\Pi_{\rm FS}(\Lambda)$ be the set of all the fully supported irreducible representations with infinitesimal character $\Lambda$. We collect the unitary members of $\Pi_{\rm FS}(\Lambda)$ as $\Pi_{\rm FS}^{\rm u}(\Lambda)$.

	\begin{lemma}\label{lemma-b2-empty}
		Let $G$ be $E_{7(7)}$. Then $\Pi_{\rm FS}^{\rm u}(\Lambda)$ is empty for any $\Lambda\in\Phi_i$ for $2\leq i\leq 14$.
	\end{lemma}
	
 Lemma \ref{lemma-b2-empty} is obtained by tedious calculations. Its validity gives an additional piece of evidence for Conjecture 1.4 of \cite{DDH}, which can be viewed as the half-integral version of Vogan's fundamental parallelepiped conjecture \cite{Vog22b}. The main task is to do non-unitarity test. To achieve this, we use the sharpened Helgason-Johnson bound, Dirac inequality, and the Vogan pencil starting from a LKT. If all these tools lose effect, we will finally use \texttt{atlas}. Let us illustrate the strategy via concrete examples.
	
\begin{example}\label{example-EVI-Phi9}
		Consider the infinitesimal character $\Lambda=[0, 0, 1, 1, 0, 1, 9]$ in $\Phi_9$.
		\begin{verbatim}
			G:E7_s
			set all=all_parameters_gamma(G, [0, 0, 1, 1, 0, 1, 9])
			#all
			Value: 5274
			set allFS=## for p in all do if #support(p)=7 then [p] else [] fi od
			#allFS
			Value: 4043
		\end{verbatim}
		Therefore, there are $5274$ irreducible representations under $\Lambda$, while $|\Pi_{\rm FS}(\Lambda)|=4043$.
		
Now let us adopt the sharpened Helgason-Johnson bound. Since it is easier to enter integers into \texttt{atlas}  than rational numbers, we collect the transposed vectors $2 \zeta_1, 2\zeta_2, \dots, 2\zeta_7$ as \texttt{TgFWts}.
\begin{verbatim}
set TgFWts=mat: [[0,1,-1,0,0,0,0],[0,1,1,0,0,0,0],[0,1,1,2,0,0,0],[0,1,1,2,
2,0,0],[0,1,1,2,2,2,0],[0,1,1,2,2,2,2],[-2,-2,-3,-4,-3,-2,-1],[2,2,3,4,3,2,1]]
\end{verbatim}
Then the command
\begin{verbatim}
(nu(p)*TgFWts)*(nu(p)*TgFWts)<4*157/2
\end{verbatim}
is equivalent to $\|\nu\|<\sqrt{\frac{157}{2}}$.
\begin{verbatim}
set oldHJ=##for p in allFS do if (nu(p)*TgFWts)*(nu(p)*TgFWts)
<=4*399/2 then [p] else [] fi od
#oldHJ
Value: 2134
set newHJ=##for p in allFS do if (nu(p)*TgFWts)*(nu(p)*TgFWts)
<4*157/2 then [p] else [] fi od
#newHJ
Value: 73
\end{verbatim}
		Therefore, there are $2134$ representations in $|\Pi_{\rm FS}(\Lambda)|$ satisfying $\|\nu\|\leq \sqrt{\frac{399}{2}}$, while only $73$ of them further meet the requirement $\|\nu\|< \sqrt{\frac{157}{2}}$.
		
Finally, let us look at one of the LKTs of the above four representations.
		
\begin{verbatim}
set x71=KGB(G,71)
void: for p in newHJ do print(highest_weight(LKTs(p)[0],x71)) od
((),KGB element #71,[1, -9,  8,  0,  0,  1,  7])
((),KGB element #71,[1, -9,  8,  0,  0,  1,  7])
		    ......
((),KGB element #71,[4, -8,  5,  1,  1,  0,  6])
\end{verbatim}
Here some \texttt{atlas} outputs are omitted to save space. Then we can compute the minimal spin norm along the Vogan pencil starting from the above LKTs and finally verify that they are all strictly smaller than the norm of $\Lambda$. Therefore, by Dirac inequality, we conclude that $\Pi_{\rm FS}^{\rm u}(\Lambda)$ is empty.
\end{example}

The above method fails on certain members of $\Phi$:
$$
[0,0,1,1,0,1,2], [0,0,1,1,0,2,0], [0,0,1,1,0,2,1], [0,0,1,1,1,0,2],
\dots, [4, 0, 0, 1, 0, 1, 0].
$$
However, a more careful look says that there is no fully supported irreducible unitary representation under them. Let us provide one example.

\begin{example}
Let us consider the infinitesimal character $\Lambda=[0,0,1,1,1,0,2]$.
\begin{verbatim}
	set all=all_parameters_gamma(G,[0,0,1,1,1,0,2])
	#all
	Value: 5274
	set allFS=## for p in all do if #support(p)=7 then [p] else [] fi od
	#allFS
	Value: 4031
\end{verbatim}
Therefore $|\Pi_{\rm FS}(\Lambda)|=4031$.

\begin{verbatim}
	set newHJ=##for p in allFS do if (nu(p)*TgFWts)*(nu(p)*TgFWts)
	<4*157/2 then [p] else [] fi od
	#nnewHJ
	Value: 2413
\end{verbatim}

A careful look at \texttt{newHJ} says that non-unitarity test using the Vogan pencil starting from a LKT fails for the following five representations:
\begin{verbatim}
newHJ[139], newHJ[510], newHJ[512], newHJ[944], newHJ[2066].
\end{verbatim}
At this stage we use \texttt{atlas}. It turns out that each of them is non-unitary. For instance,
\begin{verbatim}
is_unitary(newHJ[2066])
Value: false
\end{verbatim}
We conclude that $\Pi_{\rm FS}^{\rm u}(\Lambda)=\emptyset$.
\end{example}
	
It is interesting to note that the statistic $\|\nu\|^2$ for the $125$ FS-scattered representations of $E_{7(7)}$ is distributed as follows:
\begin{align*}
&7.5, 10.5^{\underline{4}}, 12^{\underline{4}},
12.5^{\underline{2}}, 13.5, 14.5,
16^{\underline{2}},
17^{\underline{4}},
17.5^{\underline{9}},
18^{\underline{2}},
20^{\underline{2}},
28^{\underline{8}}, 29^{\underline{18}}, 30^{\underline{18}},\\
&30.5^{\underline{3}}, 31,
40.5^{\underline{2}},
42^{\underline{8}},
53.5^{\underline{6}}, 54, 55^{\underline{7}}, 55.5^{\underline{2}}, 78^{\underline{16}}, 79.5, 115.5, 199.5.
\end{align*}
Here $a^{\underline{k}}$ means that the value $a$ occurs $k$ times.

\subsection{String representations of $E_{7(7)}$}\label{sec-string-EV}
This section aims to count the number of string representations in the Dirac series of $E_{7(7)}$ using the method of \cite{D21}. Firstly, let us verify that Conjecture 2.6 of \cite{D21} and the binary condition hold for the group $E_{7(7)}$.

\begin{example}
Consider the case that $\texttt{support(x)=[1, 2, 3, 4, 5, 6]}$. There are $1164$ such KGB elements in total. We compute that there are $59061$ infinitesimal characters $\Lambda=[a, b, c, d, e, f, g]$ in total such that
\begin{itemize}
\item[$\bullet$] $b$, $c$, $d$, $e$, $f$, $g$ are non-negative integers, $a=0$;
\item[$\bullet$] $a+c>0$, $b+d>0$, $c+d>0$, $d+e>0$, $e+f>0$ and $f+g>0$;
\item[$\bullet$] there exists a  KGB element $x$ with support $[1, 2, 3, 4, 5, 6]$ such that $\|\frac{\Lambda-\theta_x\Lambda}{2}\|< \sqrt{\frac{157}{2}} $.
\end{itemize}
Then we exhaust all the irreducible unitary representations under these infinitesimal characters  with the above $1164$ KGB elements. It turns out that such representations occur only when $b, c, d, e, f, g=0$ or $1$.  Then we check that each $\pi_{L(x)}$ is indeed unitary. \hfill\qed
\end{example}

Similarly, we handle all the other non fully supported KGB elements. It follows that Conjecture 2.6 of \cite{D21} and the binary condition hold for $E_{7(7)}$.

Now we use Section 5 of \cite{D21} to pin down the number of strings in $\widehat{E_{7(7)}}^d$.
We compute that
\begin{align*}
&N([0,1,2,4,5,6])=4, \quad N([0,1,2,3,5,6])=14, \quad N([0,1,3,4,5,6])=18, \\
&N([0,1,2,3,4,6])=30, \quad N([0,2,3,4,5,6])=74, \quad N([1,2,3,4,5,6])=112,\\
&N([0,1,2,3,4,5])=110.
\end{align*}
In particular, it follows that $N_6=362$. We also compute that
$$
N_0=72, \quad N_1=140, \quad N_2=210, \quad N_3=323, \quad N_4=454, \quad N_5=496.
$$
Therefore, the total number of strings for $E_{7(7)}$ is equal to
$$
\sum_{i=0}^{6} N_i=2057.
$$

To facilitate the classification of the Dirac series of $E_{7(7)}$, some auxiliary files have been built up. They are available via the following link:
\begin{verbatim}
https://www.researchgate.net/publication/364253102_EV-files
\end{verbatim}

\section{Dirac index}\label{sec-EV-DI}
For convenience, we assume that $G$ is equal rank in this section. Then $\dim\frp$ is even, and there is only one spin module $S_G$. Let $\pi$ be an irreducible unitary $(\frg, K)$-module.  The Dirac operator interchanges $\pi \otimes S_G^+$ and $\pi \otimes S_G^-$. Thus the Dirac cohomology $H_D(\pi)$ breaks up into the even part $H_D^+(\pi)$ and the odd part $H_D^-(\pi)$. The \emph{Dirac index} is defined as the virtual module $\widetilde{K}$-module
\begin{equation}\label{def-DI}
	{\rm DI}(\pi)=H_D^+(\pi) - H_D^-(\pi).
\end{equation}
By Proposition 2.5 of \cite{MPVZ}, we also have
$$
\text{DI}(\pi)=\pi \otimes S_G^+-\pi \otimes S_G^-.
$$
Let $\gamma$ be any $\widetilde{K}$ type of $H_D(\pi)$ (if exists).
Therefore, to compute $\text{DI}(\pi)$ from $H_D(\pi)$ via \eqref{def-DI}, it remains to detect whether $\gamma$ lives in $H_D^+(\pi)$ or $H_D^-(\pi)$. That is,  it remains to detect whether $\gamma$ lives in $\pi\otimes S^+_G$ or $\pi\otimes S^-_G$.  Lemma 2.3 of \cite{DW21} is helpful on this aspect.

Another more efficient way of computing Dirac index is given by \cite{MPVZ}. Its \texttt{atlas} realization is the following command:
	\begin{verbatim}
		show_Dirac_index(p)
	\end{verbatim}

In \cite{DDY}, we report an irreducible unitary representation $\pi$ of \texttt{F4\_s}, the linear split $F_4$, such that
$$
{\rm Hom}(H_D^+(\pi), H_D^-(\pi))\neq 0.
$$
In this case, we simply say that \emph{cancellation} happens for $\pi$.

Now for the group $E_{7(7)}$, there are $19$ FS-scattered representations in total for which cancellation happen. We will put stars on their KGB elements. In each case, the Dirac index turns out to be zero. Thus  Conjecture 1.4 of \cite{D21} holds for $E_{7(7)}$.

\begin{example}\label{exam-EV-cancellation}
Let us consider the  seventh entry of Table \ref{table-EV-1001010}, namely the representation  with \texttt{KGB(G,16648)}.  It has infinitesimal character $[1,0,0,1,0,1,0]$, which is conjugate to $\rho_c$ under the action of $W(\frg, \frh_f)$. This representation has four spin lowest $K$-types:
\begin{align*}
\rho_n^{(26)}=[2,0,2,1,0,2,2], \quad \rho_n^{(38)}=[1,0,3,0,1,2,1], \\
\rho_n^{(19)}=[1,1,1,1,1,1,3], \quad \rho_n^{(30)}=[0,1,2,0,2,1,2].
\end{align*}
Therefore, $H_D(\pi)$ consists of four copies of the trivial $\widetilde{K}$-type. Note that
\begin{align*}
&w^{(26)}=s_2s_4s_5s_6s_3s_4s_5s_1,
&w^{(38)}=s_2s_4s_5s_6s_3s_4s_5s_1s_3,\\
&w^{(19)}=s_2s_4s_5s_6s_3s_4s_1,
&w^{(30)}=s_2s_4s_5s_6s_3s_4s_1s_3.
\end{align*}
Their lengths are  $8$, $9$, $7$, $8$, respectively. Thus by  Lemma 2.3 of \cite{DW21},  there are two trivial $\widetilde{K}$-types living in $H_D^+(\pi)$, while the other two live in $H_D^-(\pi)$. As a consequence, the Dirac index of $\pi$ vanishes.

Alternatively, we can compute DI($\pi$) by \texttt{atlas} as follows (certain outputs are omitted):
\begin{verbatim}
set p=parameter(KGB(G,16648),[2,1,-2,1,1,1,0],[1,-2,-3,4,-1,2,-1])
show_dirac_index(p)
Dirac index is 0
\end{verbatim}
which agrees with the earlier calculation.\hfill\qed
\end{example}

\begin{example}\label{EV-minimal}
The first entry of Table \ref{table-EV-1110111} is the minimal representation. Its $K$-types are
$$
\{\mu_n:= n\beta\mid n\in\bbZ_{\geq 0} \}.
$$
We compute that its spin LKTs are exactly $\mu_1, \mu_2, \mu_3$ and $\mu_4$.
Moreover, both $\mu_1$ and $\mu_4$ contribute the following two $\widetilde{K}$-types:
$$
[6, 0, 0, 0, 0, 0, 0], \quad [0, 0, 0, 0, 0, 0, 6],
$$
but with opposite signs on each of them.

Both $\mu_2$ and $\mu_3$ contribute the following $18$ multiplicity-free $\widetilde{K}$-types:
\begin{align*}
&[4, 0, 1, 0, 0, 0, 1], \quad [3, 0, 1, 1, 0, 0, 0], \quad [3, 1, 0, 0, 0, 1, 1], \quad [3, 0, 0, 0, 0, 2, 1],\\
&[2, 1, 0, 1, 0, 1, 0], \quad [2, 0, 0, 1, 0, 2, 0], \quad [0, 2, 0, 1, 0, 0, 2], \quad [0, 2, 0, 0, 2, 0, 0],\\
&[0, 1, 2, 0, 0, 1, 0], \quad [0, 1, 0, 1, 0, 1, 2], \quad [1, 2, 0, 0, 0, 0, 3], \quad [0, 1, 0, 0, 2, 1, 0],\\
&[0, 0, 2, 0, 0, 2, 0], \quad [1, 0, 3, 0, 0, 0, 0], \quad [1, 1, 0, 0, 0, 1, 3], \quad [0, 0, 0, 1, 1, 0, 3],\\
&[0, 0, 0, 0, 3, 0, 1], \quad [1, 0, 0, 0, 1, 0, 4].
\end{align*}
However, their signs differ on each $\widetilde{K}$-type.
Therefore, the Dirac index of the minimal representation vanishes.

Theorem 6.2 of \cite{DW21} is also helpful here. Indeed, we note that
$$
B(\mu_n, \zeta_2)=B(n\beta, \zeta_2)=n, \quad \forall n\in\bbZ_{\geq 0}.
$$
Therefore, $\mu_1/\mu_4$, $\mu_2/\mu_3$ have distinct parities.  \hfill\qed
\end{example}

\section{Two remarkable Dirac series}\label{sec-two-ds}

For all the previously known Dirac series, say those in \cite{BDW, BP15, DD, DDH, DDL, DDY}, the spin LKTs always have multiplicity one. As we shall see, this will not hold for $E_{7(7)}$ any more.

Let us look at the fifteenth entry of Table \ref{table-EV-1001010}. Namely, the one with \texttt{KGB(G,9650)}.
\begin{verbatim}
set p=parameter(KGB(G)[9650],[1,-2,-1,4,-3,5,-1],[0,-2,-1,3,-3,4,-1])
\end{verbatim}
This representation has infinitesimal character $[1, 0, 0, 1, 0, 1, 0]$, whose norm is $\sqrt{42}$.
\begin{verbatim}
infinitesimal_character(p)*2*rho_check(G)
Value: 182/1
\end{verbatim}
As guided by Theorem \ref{thm-HP}, to study the Dirac cohomology of \texttt{p}, the output above tells us that it suffices to look at its $K$-types up to the height $182$.
\begin{verbatim}
print_branch_irr(p,KGB(G,71),182)
(1+0s)*(KGB element #71,[0, 3, 0, 0, 0, 1, 0])
(1+0s)*(KGB element #71,[0, 3, 1, 0, 0, 0, 0])
(1+0s)*(KGB element #71,[0, 2, 0, 1, 0, 0, 1])
(1+0s)*(KGB element #71,[0, 2, 0, 0, 1, 1, 0])
(1+0s)*(KGB element #71,[1, 3, 0, 0, 0, 1, 0])
(1+0s)*(KGB element #71,[1, 3, 0, 0, 0, 0, 2])
(2+0s)*(KGB element #71,[0, 2, 1, 0, 1, 0, 0])
(1+0s)*(KGB element #71,[1, 3, 1, 0, 0, 0, 0])
(2+0s)*(KGB element #71,[0, 3, 0, 0, 1, 0, 1])
(1+0s)*(KGB element #71,[0, 2, 1, 0, 0, 1, 1])
(1+0s)*(KGB element #71,[1, 4, 0, 0, 0, 0, 1])
(1+0s)*(KGB element #71,[0, 3, 0, 0, 0, 1, 2])
\end{verbatim}
Using \eqref{K-coords-shift}, we know these $K$-types are
\begin{align*}
[0, 0, 0, 0, 1, 0, 7], \quad [0, 1, 0, 0, 0, 0, 8], \quad [0, 0, 1, 0, 0, 1, 7], \quad [0, 0, 0, 1, 1, 0, 7], \\
[1, 0, 0, 0, 1, 0, 8], \quad [1, 0, 0, 0, 0, 2, 7], \quad [0, 1, 0, 1, 0, 0, 8], \quad [1, 1, 0, 0, 0, 0, 9], \\
[0, 0, 0, 1, 0, 1, 8], \quad [0, 1, 0, 0, 1, 1, 7], \quad [1, 0, 0, 0, 0, 1, 9], \quad [0, 0, 0, 0, 1, 2, 7].
\end{align*}
Their spin norms are $2 \sqrt{17},\sqrt{58},\sqrt{58},\sqrt{58},\sqrt{58},2 \sqrt{17},\sqrt{42},\sqrt{58},\sqrt{58},\sqrt{58},\sqrt{58},\sqrt{74}$, respectively.  Therefore, the seventh $K$-type is the unique spin LKT of \texttt{p}. However, the term ``\texttt{(2+0s)}" says that it has multiplicity \emph{two} in \texttt{p}. This is also \emph{partially} supported by the following:
\begin{verbatim}
show_dirac_index(p)
coeff  lambda              highest weight   fund.wt.coords.  dim
2      [1,0,0,1,0,1,0]/1   [0,0,0,0,0,0,0]  [0,0,0,0,0,0,0]  1
\end{verbatim}

The above calculation can also be checked by the \texttt{atlas} function \texttt{Dirac.at}, which is illustrated carefully in the seminar \cite{Vog22}.
\begin{verbatim}
set x=KGB(G,71)
set P=branch_irr(p,182)
<Dirac.at
set PrintDiracInfo(KTypePol P)=
void: prints(" hwt  ", "  spin wt  ");
tabulate( for c@mu in P do let (, , wt)=highest_weight(mu,x) in
[to_string(wt), to_string(DiracIC(mu))] od)
PrintDiracInfo(P)
         hwt                    spin wt
[0, 3, 0, 0, 0, 1, 0]  [1, 0, 0, 1, 1, 0, 1]/1
[0, 3, 1, 0, 0, 0, 0]  [0, 0, 1, 1, 0, 1, 0]/1
[0, 2, 0, 1, 0, 0, 1]  [0, 0, 1, 1, 0, 1, 0]/1
[0, 2, 0, 0, 1, 1, 0]  [0, 0, 1, 1, 0, 1, 0]/1
[1, 3, 0, 0, 0, 1, 0]  [0, 0, 1, 1, 0, 1, 0]/1
[1, 3, 0, 0, 0, 0, 2]  [1, 0, 0, 1, 1, 0, 1]/1
[0, 2, 1, 0, 1, 0, 0]  [1, 0, 0, 1, 0, 1, 0]/1
[1, 3, 1, 0, 0, 0, 0]  [0, 0, 1, 1, 0, 1, 0]/1
[0, 3, 0, 0, 1, 0, 1]  [0, 0, 1, 1, 0, 1, 0]/1
[0, 2, 1, 0, 0, 1, 1]  [0, 0, 1, 1, 0, 1, 0]/1
[1, 4, 0, 0, 0, 0, 1]  [0, 0, 1, 1, 0, 1, 0]/1
[0, 3, 0, 0, 0, 1, 2]  [0, 1, 1, 0, 1, 1, 0]/1
\end{verbatim}
Again, one sees that \emph{only} the \texttt{spin wt} produced by the seventh $K$-type equals the infinitesimal character $[1, 0, 0, 1, 0, 1, 0]$. Thus, it is the \emph{unique} spin LKT. By looking at \texttt{P}, one sees that this $K$-type has height $172$ and multiplicity two (the irrelevant terms are omitted below):
\begin{verbatim}
P
Value:
2* K_type(x=62, lambda=[0,1,0,1,0,0,1]/1) [172]
\end{verbatim}

We conclude that the Dirac series \texttt{p} has a unique spin LKT which has multiplicity two. Similar thing happens to the sixteenth entry of Table \ref{table-EV-1001010}, namely the one with \texttt{KGB(G,9648)}.

Let us end this section with a few remarks.

\begin{rmk}\label{rmk-mult}
(a) In the summer of 2010, J.-S. Huang told the second named author that he had announced the following conjecture at an international conference: let $\pi$ be any Dirac series. Then any $K$-type  contributing to $H_D(\pi)$ (that is, any spin LKT of $\pi$ in our language) should have multiplicity one.

Dong did not attend that conference, and there was no written form of this conjecture. Yet it was kept in his mind. After a decade, we are now able to give counter examples.

(b) The paper \cite{BP15} studies Dirac cohomology of some unipotent representations of $Sp(2n, \bbR)$ and $U(p, q)$. It is worth mentioning that Dirac cohomology there can have multiplicities. However, the underlying reason is that different spin LKTs contribute to the same $\widetilde{K}$ type, while each spin LKT there is actually multiplicity-free.

(c) The research announcement \cite{BP19} suggests application of Dirac cohomology in the theory of automorphic forms. In this direction,
we do not know the implication of multiplicities of spin LKTs.
\end{rmk}

\section{Appendix}\label{sec-appendix}
This section presents all the $125$ fully supported Dirac series of $E_{7(7)}$. See Tables \ref{table-EV-0110101}--\ref{table-EV-1111111}. Those marked with $\clubsuit$ are also special unipotent representations in the sense of \cite{BV}. They are determined for exceptional Lie groups by Adams et al. \cite{LSU}. As mentioned in Section \ref{sec-EV-DI}, those marked with stars are the Dirac series for which the even part and the odd part of the Dirac cohomology share certain $\widetilde{K}$ type(s). In each case, the Dirac index turns out to vanish completely.

There is some duality among the Dirac series of $E_{7(7)}$. For instance, in Table \ref{table-EV-0110101}, there is a FS-scattered representation $\pi$ with parameter
\begin{verbatim}
(13601, [-1,3,2,0,1,-3,5], [-1,2,2,-1,1,-4,5]).
\end{verbatim}
On the other hand, there is another FS-scattered representation $\pi^\prime$ with parameter
\begin{verbatim}
(13600, [-1,3,2,0,1,-3,5], [-1,2,2,-1,1,-4,5]).
\end{verbatim}
One can obtain the spin LKTs of $\pi^\prime$ from those of $\pi$ by reversing the coordinates. That is, the spin LKTs of $\pi^\prime$ are contragredients of those of $\pi$. Thus $\pi$ and $\pi^\prime$ are dual to each other in certain sense. In this case, we shall only present the KGB element of $\pi^\prime$ in the \textbf{bolded} fashion, and omit its other information.

\begin{table}[H]
	\centering
	\caption{Infinitesimal character $[0,1,1,0,1,0,1]$}
	\begin{tabular}{lcc}
		$\# x$ & $\lambda$/$\nu$ & Spin LKTs    \\
		\hline
		$20310$ & $[-2,1,5,-2,5,-2,5]$ &$[0,1,2,0,2,1,0]$, $[1,1,1,2,1,1,1]$, \\
		&       $[-\frac{3}{2},1,\frac{5}{2},-\frac{3}{2},\frac{5}{2},-\frac{3}{2},\frac{5}{2}]$                & $[1,0,3,0,3,0,1]$, $[0,2,2,0,2,2,0]$               \\
		$19506$ & $[-2,1,3,0,2,-1,4]$ &  $[4,0,1,0,1,0,4]$, $[3,2,0,0,0,2,3]$,\\
		&        $[-2,0,2,0,2,-2,3]$             &$[3,1,0,2,0,1,3]$, $[4,1,1,0,1,1,4]$ \\
\Xcline{1-1}{0.65pt}
		$13601$ & $[-1,3,2,0,1,-3,5]$ & 		$[0,1,2,0,0,3,2]$, $[0,2,2,1,0,2,2]$,\\
		\textbf{13600}       &$[-1,2,2,-1,1,-4,5]$ & $[1,1,1,1,1,3,1]$, $[1,0,3,0,1,2,3]$\\
\Xcline{1-1}{0.65pt}		
		$12165$ & $[-1,3,2,0,1,-3,5]$ &
		 $[1,0,0,0,2,0,7]$, $[0,1,0,1,0,2,6]$,\\
		 \textbf{12161} &$[-\frac{3}{2},2,\frac{3}{2},0,0,-\frac{7}{2},5]$ & $[0,0,2,0,1,0,7]$,  $[1,1,0,1,1,0,8]$\\
\Xcline{1-1}{0.65pt}
		 $11702$ &  $[1,0,1,-1,3,-1,2]$ &  $[0,1,2,0,1,0,5]$, $[0,2,2,0,0,0,6]$,$[1,0,2,1,0,1,5]$,\\
		\textbf{11701} &$[-3,1,4,-3,4,-3,1]$ &   $[0,2,1,0,2,0,5]$,$[0,0,6,0,0,0,2]$,$[0,2,2,0,2,0,4]$,\\
		& &  $[1,1,1,1,1,1,5]$\\
\Xcline{1-1}{0.65pt}		
		$3859$ & $[2,2,-3,2,-1,2,0]$ &
		 $[2,0,0,3,0,0,6]$, $[3,1,0,1,1,0,6]$,\\
		\textbf{3858} &$[-\frac{5}{2},1,\frac{7}{2},-\frac{5}{2},1,-1,2]$&  $[2,0,2,1,0,2,4]$, $[2,1,1,0,1,1,6]$
	\end{tabular}
	\label{table-EV-0110101}
\end{table}

\begin{table}[H]
	\centering
	\caption{Infinitesimal character $[0,1,1,0,1,1,1]$}
	\begin{tabular}{lcc}
		$\# x$ & $\lambda$/$\nu$ &   Spin LKTs   \\
		\hline
		$19804$ & $[-3,3,4,-1,1,2,1]$ &    $[0,0,0,2,0,4,0]$,$[0,1,0,1,0,5,0],$\\
		\textbf{19803} &$[-3,4,4,-3,1,1,1]$ &  $[0,0,0,3,0,4,0]$,$[0,1,0,2,0,5,0]$\\
\Xcline{1-1}{0.65pt}
		$5545$ & $[-1,1,4,-2,1,2,1]$ &   $[4,0,0,3,1,0,5]$, $[5,0,1,3,0,0,4]$,\\
		&$[-\frac{7}{2},1,\frac{9}{2},-\frac{7}{2},1,1,1]$ &  $[5,1,0,2,0,1,5]$
	\end{tabular}
	\label{table-EVI-0110111}
\end{table}

\begin{table}[H]
	\centering
	\caption{Infinitesimal character $[1,0,0,1,0,1,0]$}
	\begin{tabular}{lcc}
		$\# x$ & $\lambda$/$\nu$   & Spin LKTs   \\
		\hline
		$20895_{\clubsuit}$ & $[1,0,1,2,0,2,0]$ &    $[2,1,2,1,0,1,2$], $[2,1,2,0,2,1,0]$, $[2,2,0,1,2,0,2]$,\\
		$\textbf{20894}_{\clubsuit}$ &$[1,0,0,1,0,1,0]$ &  $[3,0,1,1,1,2,1]$, $[0,2,2,0,2,0,2]$,$[1,0,3,0,1,2,1]$,\\
		& &   $[1,1,1,1,1,1,3]$, $[1,1,1,0,3,1,1]$\\
\Xcline{1-1}{0.65pt}		
		$20854_{\clubsuit}$ & $[1,0,0,3,0,2,0]$  & $[0,1,0,2,1,0,5]$, $[2,0,1,0,2,0,5]$,\\
		$\textbf{20853}_{\clubsuit}$ &$[1,0,0,1,0,1,0]$ &  $[1,2,0,1,0,1,5]$, $[1,1,1,1,1,1,3]$\\
\Xcline{1-1}{0.65pt}		
		$20274_{\clubsuit}$ & $[3,1,1,-1,1,2,0]$  & $[1,1,0,1,1,0,6]$, $[1,0,1,1,2,0,4]$,\\
		$\textbf{20273}_{\clubsuit}$ &$[2,-\frac{1}{2},-\frac{1}{2},1,-\frac{1}{2},2,0]$   &   $[0,2,0,2,0,1,4]$, $[2,1,1,0,1,1,4]$\\
\Xcline{1-1}{0.65pt}		
		$16648^*$ & $[2,1,-2,1,1,1,0]$ &   $[2,0,2,1,0,2,2]$, $[1,0,3,0,1,2,1]$,\\
		$\textbf{16647}^*$ &$[1,-2,-3,4,-1,2,-1]$ &  $[1,1,1,1,1,1,3]$, $[0,1,2,0,2,1,2]$\\
\Xcline{1-1}{0.65pt}		
		$15450^*$ & $[2,-1,-4,5,0,2,-1]$ &   $[0,1,0,1,0,0,8]$, $[1,0,0,2,0,0,7]$,\\
		$\textbf{15446}^*$ &$[\frac{3}{2},-\frac{3}{2},-\frac{7}{2},\frac{7}{2},0,\frac{3}{2},-\frac{3}{2}]$ &   $[0,2,0,0,1,0,7]$, $[1,1,0,1,1,0,6]$\\
\Xcline{1-1}{0.65pt}		
		$15187$ & $[4,-2,-2,4,-2,4,-2]$  & $[0,0,6,0,0,0,0]$, $[0,4,0,0,0,0,6]$, $[1,2,0,1,0,1,5]$,\\
		$\textbf{15186}$ &$[3,-2,-2,3,-2,3,-2]$ &   $[2,0,1,0,2,0,5]$, $[0,1,0,2,1,0,5]$,$[1,1,1,1,1,1,3]$,\\
		& & $[0,2,2,0,2,0,2]$\\
\Xcline{1-1}{0.65pt}	
        $11818$ & $[7,0,-4,2,0,2,0]$  & $[0,1,2,0,2,1,2]$, $[2,0,2,1,0,2,2]$,\\
	    $\textbf{11817}$ &$[\frac{9}{2},0,-\frac{7}{2},1,0,1,0]$ &   $[1,2,1,1,1,0,3]$, $[1,1,3,0,1,1,1]$\\
\Xcline{1-1}{0.65pt}
         $9650$ & $[1,-2,-1,4,-3,5,-1]$  & $[0,1,0,1,0,0,8]$ (multiplicity two) \\
         $\textbf{9648}$      &$[0,-2,-1,3,-3,4,-1]$ & \\	
\Xcline{1-1}{0.65pt}
	    $8645$ & $[5,0,-4,3,0,1,0]$  & $[0,0,1,0,0,0,9]$, $[2,0,1,0,2,0,5]$,\\
	    $\textbf{8641}$ &$[4,0,-4,2,0,0,0]$ &  $[4,0,3,0,0,0,3]$\\
\Xcline{1-1}{0.65pt}
	    $5879$ & $[2,0,-1,2,-1,2,-1]$ &  $[0,4,0,0,0,0,6]$, $[1,2,0,1,0,1,5]$, $[2,0,1,0,2,0,5]$,\\
	    $\textbf{5878}$ &$[1,-2,-2,3,-2,3,-2]$ &    $[0,1,0,2,1,0,5]$, $[1,1,1,1,1,1,3]$, $[0,2,2,0,2,0,2]$\\
\Xcline{1-1}{0.65pt}
	    $4832$ & $[1,0,0,1,1,0,-1]$ &   $[1,3,0,1,0,0,5]$, $[2,1,1,0,1,1,4]$,\\
	    $\textbf{4831}$ &$[1,-1,-1,2,-\frac{5}{2},\frac{7}{2},-\frac{5}{2}]$ &   $[0,2,0,2,0,1,4]$, $[1,2,1,1,1,0,3]$
	
	\end{tabular}
	\label{table-EV-1001010}
\end{table}

\begin{table}[H]
	\centering
	\caption{Infinitesimal character $[1,0,0,1,0,1,1]$}
	\begin{tabular}{lcc}
		$\# x$ & $\lambda$/$\nu$   & Spin LKTs   \\
		\hline
		$20925_{\clubsuit}$ & $[1,1,1,1,1,1,1]$ &    $[2,1,1,1,1,1,2]$\\ & $[1,0,0,1,0,1,1]$ &\\
		$20891_{\clubsuit}$ & $[0,0,0,4,0,0,3]$ &    $[2,1,1,1,1,1,2]$\\ & $[1,0,0,1,0,1,1]$& \\
		$20787_{\clubsuit}$ & $[4,-1,-1,4,-1,4,1]$  & $[1,1,2,0,2,1,1]$\\ & $[\frac{3}{2},-\frac{1}{2},-\frac{1}{2},\frac{3}{2},-\frac{1}{2},\frac{3}{2},1]$ &\\
		$20424_{\clubsuit}$ & $[3,0,0,1,0,3,1]$  & $[3,1,1,0,1,1,3]$\\& $[2,-\frac{1}{2},-\frac{1}{2},1,-\frac{1}{2},2,1]$& \\
\Xcline{1-1}{0.65pt}		
        $18488^*$ & $[2,1,-3,3,0,0,2]$  & $[1,1,1,0,1,1,5]$, $[0,1,1,1,1,1,4]$  \\
        $\textbf{18487}^*$ &$[1,-3,-3,4,0,1,1]$& \\
\Xcline{1-1}{0.65pt}		
		$16431^*$ & $[3,-6,1,9,-6,1,3]$  & $[1,1,1,1,2,2,0]$, $[0,2,2,1,1,1,1]$,\\
		&$[1,-\frac{7}{2},0,\frac{9}{2},-\frac{7}{2},1,1]$ &   $[2,1,0,2,1,2,1]$, $[1,2,1,2,0,1,2]$\\
\Xcline{1-1}{0.65pt}
		$13623^*$ & $[2,-1,-1,3,-1,1,1]$ &   $[0,1,2,1,1,2,1]$, $[1,1,1,2,0,2,2]$\\
        $\textbf{13622}^*$ &$[1,-4,-1,5,-3,1,1]$&\\
\Xcline{1-1}{0.65pt}		
		$13442$ & $[3,-6,-4,7,-4,9,1]$  & $[2,1,1,0,2,2,1]$, $[1,2,2,0,1,1,2]$\\ &$[1,-\frac{5}{2},-\frac{5}{2},\frac{7}{2},-\frac{5}{2},\frac{7}{2},1]$&\\
\Xcline{1-1}{0.65pt}
		$12480$ & $[-1,-1,0,3,-2,3,2]$  & $[1,1,2,0,1,2,2]$\\
        $\textbf{12479}$ &$[1,-2,-2,3,-3,4,1]$&\\
\Xcline{1-1}{0.65pt}		
		$12195^*$ & $[3,-3,-2,6,-3,2,1]$  & $[1,1,0,0,1,0,8]$, $[2,0,0,1,1,0,7]$\\
        $\textbf{12191}^*$ &$[\frac{3}{2},-\frac{7}{2},-\frac{3}{2},5,-\frac{7}{2},2,0]$&\\
\Xcline{1-1}{0.65pt}		
		$10855$ & $[1,-2,-1,4,-3,4,2]$  & $[0,1,0,1,1,0,7]$\\
        $\textbf{10853}$ &$[0,-2,-\frac{3}{2},\frac{7}{2},-\frac{7}{2},\frac{7}{2},\frac{3}{2}]$&\\
\Xcline{1-1}{0.65pt}
        $8176^*$ & $[1,-1,-3,6,-3,2,1]$  & $[0,2,2,1,1,0,3]$, $[1,1,3,1,0,1,2]$\\
$\textbf{8175}^*$ &$[1,-\frac{1}{2},-3,\frac{9}{2},-\frac{7}{2},1,1]$&\\
\Xcline{1-1}{0.65pt}

		$7776$ & $[3,-3,-2,5,-2,1,2]$  & $[3,1,0,1,1,1,4]$, $[4,1,1,1,0,1,3]$,\\
		&$[\frac{7}{2},-\frac{5}{2},-\frac{5}{2},\frac{7}{2},-\frac{5}{2},1,1]$ &  $[2,0,0,2,2,0,4]$, $[4,0,2,2,0,0,2]$\\

		$7038$ & $[4,-2,-3,4,-1,1,2]$ &   $[1,0,3,0,2,2,0]$, $[0,2,2,0,3,0,1]$\\
               &$[4,-2,-3,3,-2,1,1]$&\\
		$7037$ & $[4,-2,-3,4,-1,1,2]$  & $[3,0,0,1,2,0,5]$, $[5,0,2,1,0,0,3]$\\
               &$[4,-2,-3,3,-2,1,1]$&\\
		$7036$ & $[4,-2,-3,4,-1,1,2]$  & $[2,1,0,2,1,1,3]$, $[3,1,1,2,0,1,2]$,\\
		       &$[4,-2,-3,3,-2,1,1]$  &  \\

\Xcline{1-1}{0.65pt}		
        $7486^*$ & $[1,0,-2,5,-4,3,1]$ &   $[1,0,1,1,0,0,8]$, $[1,1,1,0,1,0,7]$\\
        $\textbf{7482}^*$ &$[1,0,-\frac{5}{2},4,-4,\frac{3}{2},1]$&\\
\Xcline{1-1}{0.65pt}		
		$3623$ & $[3,0,-2,3,-2,1,3]$  & ${\rm LKT}=[0,1,0,0,0,0,10]$\\
        $\textbf{3619}$ &$[\frac{5}{2},0,-\frac{5}{2},\frac{5}{2},-\frac{5}{2},0,\frac{5}{2}]$&\\
\Xcline{1-1}{0.65pt}		
		$2860$ & $[2,-1,-2,4,-2,1,2]$  & $[4,0,1,0,2,1,3]$, $[3,1,2,0,1,0,4]$,\\
		&$[1,-2,-2,3,-2,1,1]$ &  $[2,0,2,1,1,1,3]$, $[3,1,1,1,2,0,2]$   		
	\end{tabular}
	\label{table-EV-1001011}
\end{table}

\begin{table}[H]
	\centering
	\caption{Infinitesimal character $[1,0,0,1,1,0,1]$}
	\begin{tabular}{lcc}
		$\# x$ & $\lambda$/$\nu$ & Spin LKTs   \\
		\hline
		$11317$ & $[1,-2,-1,4,1,-3,5]$   & $[0,2,2,0,0,3,2]$,$[1,1,2,0,1,3,2]$\\
        $\textbf{11316}$ &$[1,-2,-2,3,1,-4,5]$&\\
\Xcline{1-1}{0.65pt}		
		$9656$ & $[1,-1,-1,3,1,-2,4]$  & $[0,0,1,0,2,0,7]$, $[0,1,0,1,1,1,7]$\\
        $\textbf{9654}$ & $[0,-2,-\frac{3}{2},\frac{7}{2},0,-\frac{7}{2},5]$ &
	\end{tabular}
	\label{table-EV-1001101}
\end{table}

\begin{table}[H]
	\centering
	\caption{Infinitesimal character $[1,0,0,1,1,1,1]$}
	\begin{tabular}{lcc}
		$\# x$ & $\lambda$/$\nu$ & Spin LKTs   \\
		\hline
		$18856$ & $[2,1,-1,2,-1,1,2]$   & $[0,0,0,2,0,5,0]$\\
        $\textbf{18855}$ & $[1,-4,-4,5,1,1,1]$&		
	\end{tabular}
	\label{table-EV-1001111}
\end{table}

\begin{table}[H]
	\centering
	\caption{Infinitesimal character $[1,0,1,1,0,1,0]$}
	\begin{tabular}{lcc}
		$\# x$ & $\lambda$/$\nu$ & Spin LKTs   \\
		\hline
		$20895_{\clubsuit}$ & $[1,0,1,2,0,2,0]$   & $[0,2,0,2,0,3,0]$\\
        $\textbf{20894}_{\clubsuit}$ & $[1,0,1,1,0,1,0]$&		
	\end{tabular}
	\label{table-EV-1011010}
\end{table}

\begin{table}[H]
	\centering
	\caption{Infinitesimal character $[1,0,1,1,0,1,1]$}
	\begin{tabular}{lcc}
		$\# x$ & $\lambda$/$\nu$ & Spin LKTs   \\
		\hline
		$16431$ & $[3,-5,1,8,-5,1,3]$   & $[0,0,2,3,2,0,0]$\\&$[1,-\frac{9}{2},1,\frac{11}{2},-\frac{9}{2},1,1]$&\\
		$13900$ & $[1,-1,1,4,-3,2,1]$   & $[6,0,0,1,0,0,6]$\\& $[0,-4,2,5,-5,2,0]$&		
	\end{tabular}
	\label{table-EV-1011011}
\end{table}

\begin{table}[H]
	\centering
	\caption{Infinitesimal character $[1,1,0,1,0,1,0]$}
	\begin{tabular}{lcc}
		$\# x$ & $\lambda$/$\nu$ & Spin LKTs   \\
		\hline
		$19205$ & $[3,3,-6,9,-6,9,-4]$   & $[0,1,2,1,2,1,0]$\\ &$[1,1,-\frac{5}{2},\frac{7}{2},-\frac{5}{2},\frac{7}{2},-\frac{5}{2}]$&\\
		$17644$ & $[1,1,-2,5,-2,3,-2]$   & $[4,1,0,1,0,1,4]$\\& $[0,0,-2,4,-2,3,-3]$&\\
		$5492$ & $[5,3,-2,1,-3,6,-3]$   & $[2,1,1,2,0,3,1]$, $[1,3,0,2,1,1,2]$\\&$[3,1,-2,1,-\frac{5}{2},\frac{7}{2},-\frac{5}{2}]$&\\
		$3923$ & $[3,1,0,1,-2,3,0]$   & $[3,2,0,0,1,2,4]$, $[4,2,1,0,0,2,3]$\\& $[3,0,-2,2,-3,3,-2]$&		
	\end{tabular}
	\label{table-EV-1101010}
\end{table}

\begin{table}[H]
	\centering
	\caption{Infinitesimal character $[1,1,0,1,0,1,1]$}
	\begin{tabular}{lcc}
		$\# x$ & $\lambda$/$\nu$  & Spin LKTs   \\
		\hline
		$20467$ & $[3,1,0,1,0,1,1]$   & $[0,0,0,2,0,3,0]$,$[0,2,0,0,0,5,0],$ $[0,1,0,2,0,4,0]$,\\
		$\textbf{20466}$ &$[3,1,-2,3,-2,1,1]$ & $[0,0,0,4,0,3,0],$ $[0,2,0,2,0,5,0]$\\
		\Xcline{1-1}{0.65pt}
		$10205$ & $[2,2,0,-1,1,2,1]$   & ${\rm LKT}=[0,0,0,4,0,3,0]$, $[0,1,1,3,0,3,1]$,\\
		$\textbf{10204}$ &$[5,1,-4,1,-2,3,1]$  &$[0,2,0,2,0,5,0]$\\
		\Xcline{1-1}{0.65pt}
		$9507$ & $[1,1,-1,3,-2,2,1]$   & ${\rm LKT}=[0,0,4,0,0,0,6]$, $[0,1,3,0,1,0,6]$,\\
		$\textbf{9506}$ &$[1,1,-4,5,-4,1,1]$ &$[0,0,6,0,0,0,4]$, $[0,2,2,0,2,0,6]$\\
		\Xcline{1-1}{0.65pt}
		$8741$ & $[4,2,-2,1,-1,3,1]$   & $[3,0,0,0,1,1,8]$, $[4,0,1,0,0,0,9]$,\\
		$\textbf{8737}$ &$[5,\frac{3}{2},-\frac{7}{2},0,-\frac{3}{2},\frac{7}{2},0]$  &$[2,0,1,0,2,0,9]$	
	\end{tabular}
	\label{table-EV-1101011}
\end{table}

\begin{table}[H]
	\centering
	\caption{Infinitesimal character $[1,1,0,1,1,0,1]$}
	\begin{tabular}{lcc}
		$\# x$ & $\lambda$/$\nu$ & Spin LKTs   \\
		\hline
		$9002$ & $[0,2,2,-1,1,1,0]$   & ${\rm LKT}=[0,0,0,4,0,4,0]$,$[0,0,1,4,0,3,1],$\\
		$\textbf{9001}$ &$[5,1,-4,1,1,-3,4]$  & $[0,1,0,3,0,5,0]$,$[0,1,1,3,0,4,1]$\\
\Xcline{1-1}{0.65pt}		
		$7588$ & $[4,2,-2,1,2,-2,3]$   & ${\rm LKT}=[4,0,0,0,0,2,8]$, $[3,0,0,0,1,2,8]$,\\
		$\textbf{7584}$ &$[5,\frac{3}{2},-\frac{7}{2},0,2,-\frac{7}{2},\frac{7}{2}]$  &$[4,0,1,0,0,1,9]$, $[3,0,1,0,1,1,9]$	
	\end{tabular}
	\label{table-EV-1101101}
\end{table}

\begin{table}[H]
	\centering
	\caption{Infinitesimal character $[1,1,0,1,1,1,1]$}
	\begin{tabular}{lcc}
		$\# x$ & $\lambda$/$\nu$ & Spin LKTs   \\
		\hline
		$14500$ & $[8,1,-5,1,3,1,3]$   & ${\rm LKT}=[0,0,0,7,0,0,0]$,$[0,0,1,6,1,0,0]$\\
        &$[\frac{15}{2},1,-\frac{13}{2},1,1,1,1]$&\\
        $11909$ & $[4,1,-3,2,1,2,1]$   & ${\rm LKT}=[8,0,0,0,0,0,8]$,$[8,0,0,1,0,0,8]$\\
        &$[7,1,-7,2,0,2,0]$&
	\end{tabular}
	\label{table-EV-1101111}
\end{table}

\begin{table}[H]
	\centering
	\caption{Infinitesimal character $[1,1,1,0,1,0,1]$}
	\begin{tabular}{lcc}
		$\# x$ & $\lambda$/$\nu$ & Spin LKTs   \\
		\hline
		$20925_{\clubsuit}^*$ & $[1,1,1,1,1,1,1]$   &$[0,2,0,2,0,2,0]$, $[0,3,0,1,0,3,0]$, \\
		&$[1,1,1,0,1,0,1]$ &$[0,2,0,3,0,2,0]$, $[0,3,0,2,0,3,0]$ \\
\Xcline{1-1}{0.65pt}
		$20771$ & $[1,2,0,-1,3,-1,2]$   & $[0,1,0,2,0,3,0]$,$[0,2,0,1,0,4,0],$\\
		$\textbf{20770}$ &$[1,1,2,-1,2,-1,1]$ & $[0,1,0,3,0,3,0]$,$[0,2,0,2,0,4,0]$\\
\Xcline{1-1}{0.65pt}
		
		$18012$ & $[3,8,1,-5,8,-5,3]$   & $[0,0,2,2,2,0,0]$, $[0,1,1,3,1,1,0]$,\\
		&$[1,\frac{9}{2},1,-\frac{7}{2},\frac{9}{2},-\frac{7}{2},1]$  &$[0,0,3,1,3,0,0]$, $[0,1,2,2,2,1,0]$\\

		$16062$ & $[1,3,1,-2,5,-2,1]$   & $[5,0,0,1,0,0,5]$, $[5,1,0,0,0,1,5]$,\\
		&$[0,4,2,-4,5,-3,0]$ &$[5,0,0,2,0,0,5]$, $[5,1,0,1,0,1,5]$\\

\Xcline{1-1}{0.65pt}
		$16308^*$ & $[2,2,1,0,1,-2,4]$   & $[0,0,2,0,0,3,4]$, $[0,0,3,0,0,2,5]$,\\
		$\textbf{16307}^*$ &$[1,1,1,0,1,-5,6]$ &$[0,1,1,0,1,3,4]$, $[0,1,2,0,1,2,5]$\\
\Xcline{1-1}{0.65pt}		

		$11170$ & $[3,1,1,-4,9,-6,7]$   & $[1,3,1,0,1,3,1]$, $[0,4,2,0,0,2,2]$,\\
		&$[1,1,1,-\frac{7}{2},\frac{9}{2},-\frac{7}{2},\frac{9}{2}]$ &$[2,2,0,0,2,4,0]$\\

\Xcline{1-1}{0.65pt}
		$10225$ & $[1,1,5,-2,1,-2,5]$   & $[0,0,1,3,0,3,1]$, $[0,1,1,2,0,4,1]$,\\
		$\textbf{10224}$ &$[1,1,4,-3,1,-3,4]$ &$[0,0,2,3,0,2,2]$, $[0,1,2,2,0,3,2]$\\
\Xcline{1-1}{0.65pt}
		$10213$ & $[-1,2,1,0,1,0,2]$   & $[0,3,2,0,0,3,2]$, $[1,2,1,0,1,4,1]$,\\
		$\textbf{10212}$ &$[1,1,1,-3,4,-4,5]$  &\\
\Xcline{1-1}{0.65pt}
$8753$ & $[2,2,3,-2,2,-2,3]$  & ${\rm LKT}=[3,0,0,0,0,2,7]$, $[2,0,0,0,1,2,7]$,\\
		$\textbf{8749}$ &$[\frac{3}{2},\frac{3}{2},\frac{7}{2},-\frac{7}{2},2,-\frac{7}{2},\frac{7}{2}]$ & $[3,0,1,0,0,1,8]$, $[2,0,1,0,1,1,8]$\\
\Xcline{1-1}{0.65pt}
		
		$8603$ & $[1,2,2,-2,3,-2,4]$   & ${\rm LKT}=[0,0,0,0,3,0,7]$, $[0,1,0,1,2,0,8]$\\
$\textbf{8601}$&$[0,\frac{3}{2},2,-\frac{7}{2},\frac{7}{2},-\frac{7}{2},5]$&\\
\Xcline{1-1}{0.65pt}		
		$6066$ & $[1,-3,2,0,2,1,0]$   & ${\rm LKT}=[0,0,3,0,0,1,7]$, $[1,0,2,0,0,2,7]$,\\
		$\textbf{6065}$ &$[1,\frac{9}{2},1,-\frac{7}{2},1,-\frac{3}{2},\frac{5}{2}]$  &$[0,1,2,1,0,0,8]$, $[1,1,1,1,0,1,8]$\\
\Xcline{1-1}{0.65pt}		
		$5421$ & $[2,4,2,-3,2,-1,3]$   & ${\rm LKT}=[4,0,0,0,4,2,0]$, $[3,0,0,1,3,3,0]$,\\
		$\textbf{5420}$ &$[1,4,\frac{3}{2},-4,\frac{3}{2},-\frac{3}{2},\frac{5}{2}]$  &$[4,0,1,0,4,1,1]$, $[3,0,1,1,3,2,1]$\\
\Xcline{1-1}{0.65pt}		
		$4009$ & $[1,1,2,-2,4,-2,2]$   & {\rm LKT}=$[0,3,0,0,3,0,5]$, $[0,4,0,0,2,0,6]$,\\
		$\textbf{4008}$ &$[1,1,1,-3,4,-3,1]$ &$[0,3,1,0,3,0,4]$, $[1,2,0,1,2,1,5]$
	\end{tabular}
	\label{table-EV-1110101}
\end{table}

\begin{table}[H]
	\centering
	\caption{Infinitesimal character $[1,1,1,0,1,1,1]$}
	\begin{tabular}{lcc}
		$\# x$ & $\lambda$/$\nu$ & Spin LKTs   \\
		\hline
		$20925_{\clubsuit}^*$ & $[1,1,1,1,1,1,1]$   &$[0,0,0,3,0,0,0]+n\beta$, $0\leq n\leq 3$ \\
		&$[1,1,1,0,1,1,1]$ & \\
\Xcline{1-1}{0.65pt}
		$18324$ & $[1,1,1,-2,5,1,1]$   & $[0,0,0,0,0,6,0]+n\beta$, $0\leq n\leq 3$;\\
		$\textbf{18323}$ &$[1,1,1,-5,6,1,1]$ & ${\rm LKT}=[0,0,0,0,0,6,0]$\\
\Xcline{1-1}{0.65pt}
		
		$15505$ & $[1,1,5,-2,1,1,1]$   & ${\rm LKT}=[0,0,0,6,0,0,0]$, \\
		&$[1,1,\frac{13}{2},-\frac{11}{2},1,1,1]$  &$[0,0,1,5,1,0,0]$, $[0,0,2,4,2,0,0]$\\
		$12885$ & $[1,1,3,-1,1,1,1]$   & $[7,0,0,0,0,0,7]+n\beta$, $0\leq n\leq 2$;\\
		&$[0,1,7,-5,0,2,0]$  & ${\rm LKT}=[7,0,0,0,0,0,7]$\\
\Xcline{1-1}{0.65pt}
		$8375$ & $[1,4,1,-2,2,1,1]$   & ${\rm LKT}=[0,0,5,0,0,0,7]$, \\
		$\textbf{8374}$ &$[1,6,1,-5,1,1,1]$ &$[0,0,6,0,0,0,6]$, $[0,1,4,0,1,0,7]$
	\end{tabular}
	\label{table-EV-1110111}
\end{table}

\begin{table}[H]
	\centering
	\caption{Infinitesimal character $[1,1,1,1,0,1,0]$}
	\begin{tabular}{lcc}
		$\# x$ & $\lambda$/$\nu$ & Spin LKTs   \\
		\hline
		$8819$ & $[1,2,1,1,-2,4,-1]$   & $[1,4,0,0,1,5,0]$,$[0,5,1,0,0,4,1]$\\&$[1,1,1,1,-\frac{9}{2},\frac{11}{2},-\frac{9}{2}]$&
	\end{tabular}
	\label{table-EV-1111010}
\end{table}

\begin{table}[H]
	\centering
	\caption{Infinitesimal character $[1,1,1,1,0,1,1]$}
	\begin{tabular}{lcc}
		$\# x$ & $\lambda$/$\nu$ & Spin LKTs   \\
		\hline
		$17619$ & $[1,1,2,-1,2,0,1]$   & $[0,0,0,0,0,7,0]+n\beta$, $0\leq n\leq 2$;\\
		$\textbf{17618}$ &$[1,1,1,1,-6,7,1]$ & ${\rm LKT}=[0,0,0,0,0,7,0]$
	\end{tabular}
	\label{table-EV-1111011}
\end{table}

\begin{table}[H]
	\centering
	\caption{Infinitesimal character $[1,1,1,1,1,0,1]$}
	\begin{tabular}{lcc}
		$\# x$ & $\lambda$/$\nu$ & Spin LKTs   \\
		\hline
		$16791$ & $[2,1,1,1,1,-1,3]$   & ${\rm LKT}=[0,0,0,0,0,8,0]$,\\
        $\textbf{16790}$&$[1,1,1,1,1,-7,8]$& $[0,0,0,1,0,8,0]$
	\end{tabular}
	\label{table-EV-1111101}
\end{table}

\begin{table}[H]
\centering
\caption{Infinitesimal character $[1,1,1,1,1,1,1]$}
\begin{tabular}{lcc}
$\# x$ & $\lambda$/$\nu$ & Spin LKT   \\
\hline	
$20925_{\clubsuit}$ & $[1,1,1,1,1,1,1]$   & ${\rm LKT}=[0,0,0,0,0,0,0]$\\
                    & $[1,1,1,1,1,1,1]$
\end{tabular}
\label{table-EV-1111111}
\end{table}

\centerline{\scshape Acknowledgements}
We are deeply grateful to the \texttt{atlas} mathematicians. We thank an anonymous referee sincerely for giving us suggestions.

\centerline{\scshape Funding}
Dong is supported by the National Natural Science Foundation of China (grant 12171344).


\begin{thebibliography}{99}

\bibitem{ALTV} J.~Adams, M.~van Leeuwen, P.~Trapa, and D.~Vogan, \emph{Unitary representations of real reductive groups}, Ast{\'erisque} \textbf{417} (2020).

\bibitem{BDW} D.~Barbasch, C.-P.~Dong, K.D.~Wong, \emph{Dirac series for complex classical Lie groups: A multiplicity one theorem}, Adv. Math. \textbf{403} (2022), Paper Number 108370, 47pp.


\bibitem{BP15} D.~Barbasch, P.~Pand\v zi\'c, \emph{Dirac cohomology of unipotent representations of
$Sp(2n, \bbR)$ and $U(p, q)$}, J. Lie Theory \textbf{25} (1) (2015), 185--213.

\bibitem{BP19} D.~Barbasch, P.~Pand\v zi\'c,  \emph{Twisted Dirac index and applications
to characters}, Affine, vertex and $W$-algebras, pp.~23--36, Springer INdAM Series \textbf{37}, 2019.


\bibitem{BV} D.~Barbasch, D.~Vogan,
\emph{Unipotent representations of complex semisimple Lie groups},
Ann. of Math. \textbf{121} (1985), 41--110.

\bibitem{Ca} J.~Carmona, \emph{Sur la classification des modules admissibles irr\'eductibles}, pp.11--34 in Noncommutative Harmonic Analysis and Lie Groups, J. Carmona and M. Vergne, eds., Lecture
Notes in Mathematics \textbf{1020}, Springer-Verlag, New York, 1983.

\bibitem{DD}  Y.-H.~Ding,  C.-P.~Dong,
\emph{Dirac series for $E_{7(-25)}$}, J. Algebra \textbf{614} (2023), 670--694.

\bibitem{DDH}  L.-G.~Ding,  C.-P.~Dong and H.~He,
\emph{Dirac series for $E_{6(-14)}$}, J. Algebra \textbf{590} (2022), 168--201.

\bibitem{DDL}  Y.-H.~Ding,  C.-P.~Dong and P.-Y.~Li,
\emph{Dirac series for $E_{7(-5)}$}, Indag. Math., https://doi.org/10.1016/j.indag.2022.09.002.

\bibitem{DDY}  J.~Ding,   C.-P.~Dong and L.~Yang,
\emph{Dirac series for some real exceptional Lie groups}, J. Algebra \textbf{559} (2020), 379--407.

\bibitem{D13}  C.-P.~Dong,
\emph{On the Dirac cohomology of complex Lie group representations},
Transform. Groups \textbf{18} (1) (2013), 61--79. [Erratum: Transform.
Groups \textbf{18} (2) (2013), 595--597.]

\bibitem{D17} C.-P.~Dong,
\emph{Unitary representations with Dirac cohomology: finiteness in the real case},
Int. Math. Res. Not. IMRN \textbf{2020} (24), 10277--10316.

\bibitem{D20} C.-P.~Dong,
\emph{On the Helgason-Johnson bound},
Israel J. Math., https://doi.org/10.1007/s11856-022-2403-6.

\bibitem{D21} C.-P.~Dong,
\emph{A non-vanishing criterion for Dirac cohomology}, Transformation groups, https://doi.org/10.1007/s00031-022-09758-0.

\bibitem{DW21} C.-P.~Dong and K.D.~Wong,
\emph{Dirac index of some unitary representations of $Sp(2n, \bbR)$ and $SO^*(2n)$}, J. Algebra \textbf{603} (2022), 1--37.

\bibitem{HJ} S.~Helgason, K.~Johnson, \emph{The bounded spherical functions on symmetric spaces}, Adv. Math.
\textbf{3} (1969), 586--593.


\bibitem{HP} J.-S. Huang and P.~Pand\v zi\'c, \emph{Dirac
	cohomology, unitary representations and a proof of a conjecture of
	Vogan}, J. Amer. Math. Soc.  \textbf{15} (2002), 185--202.

\bibitem{HP2} J.-S. Huang and P.~Pand\v zi\'c, \emph{Dirac operators in Representation Theory}, Math. Theory Appl., Birkhauser, 2006.

\bibitem{Kn} A.~Knapp, \emph{Lie Groups, Beyond an Introduction}, Birkh\"{a}user, 2nd Edition, 2002.

\bibitem{Ko} B.~Kostant, \emph{Lie algebra cohomology and the generalized Borel-Weil theorem},
Ann. of Math. \textbf{74} (1961), 329--387.


\bibitem{MPVZ} S.~Mehdi, P.~Pand\v zi\'c,  D.~Vogan and R.~Zierau, \emph{Dirac index and associated cycles of Harish-Chandra modules}, Adv. Math. \textbf{361} (2020), 106917, 34 pp.


\bibitem{Pa1} R.~Parthasarathy, \emph{Dirac operators and the discrete
series}, Ann. of Math. \textbf{96} (1972), 1--30.

\bibitem{Pa2} R.~Parthasarathy, \emph{Criteria for the unitarizability of some highest weight modules},
Proc. Indian Acad. Sci. \textbf{89} (1) (1980), 1--24.

\bibitem{PRV} K.~R.~Parthasarathy, R.~Ranga Rao, and
S.~Varadarajan, \emph{Representations of complex semi-simple Lie
groups and Lie algebras}, Ann. of Math. \textbf{85} (1967),
383--429.

\bibitem{Sa} S.~Salamanca-Riba, \emph{On the unitary dual of real reductive Lie groups and the $A_{\mathfrak{q}}(\lambda)$ modules: the
	strongly regular case}, Duke Math. J. \textbf{96} (3) (1999),  521--546.


\bibitem{SV} S.~Salamanca-Riba, D.~Vogan, \emph{On the classification of unitary representations of reductive Lie
	groups}, Ann. of Math. \textbf{148} (3) (1998), 1067--1133.

\bibitem {Vog80} D.~Vogan,
\emph{Singular unitary representations},  Noncommutative harmonic
analysis and Lie groups (Marseille, 1980),  506--535.

\bibitem {Vog81} D.~Vogan, \emph{Representations of real reductive Lie groups}, Birkh\"auser, 1981.


\bibitem{Vog97} D.~Vogan, \emph{Dirac operators and unitary
representations}, 3 talks at MIT Lie groups seminar, Fall 1997.

\bibitem{Vog22} D.~Vogan, \emph{Dirac operator in atlas},
\texttt{atlas} seminar, June 23, 2022.

\bibitem{Vog22b} D.~Vogan, \emph{Classifying the unitary dual (part 1 of infinitely many...)}, \texttt{atlas} seminar, June 30, 2022.

\bibitem{At} Atlas of Lie Groups and Representations, version 1.1, July 2022. See www.liegroups.org for more about the software.

\bibitem{LSU} A complete list of special unipotent representations of real exceptional groups, see  http://www.liegroups.org/tables/unipotentExceptional/.
\end{thebibliography}
\end{document}